\documentclass[11pt,reqno]{article}

\usepackage[utf8]{inputenc}
\usepackage[english]{babel}

\usepackage{csquotes}

\usepackage{geometry}  
\usepackage{graphicx}
\usepackage{amssymb}
\usepackage{epstopdf}
\usepackage{amsmath}
\usepackage{amssymb}
\usepackage{epsfig}
\usepackage{graphicx}
\usepackage{changebar}
\usepackage{centernot}
\usepackage{rotating}
\usepackage{amsthm}
\usepackage{caption}
\usepackage{rotating}
\usepackage{tikz}
\usepackage{setspace}
\usepackage{subfigure}
\usepackage{changepage}
\usepackage{mathtools}
\usepackage{mathrsfs}
\usepackage{youngtab}
\usepackage{color}
\usepackage{mdframed}
\usepackage{bm}
\usepackage{fancyhdr}
\usepackage{lastpage}
\usepackage{txfonts}
\usepackage{slashed}
\usepackage{comment}
\usepackage{longtable}
\usepackage[hang,flushmargin]{footmisc}

\usepackage{blkarray}
\usetikzlibrary{arrows,patterns,snakes}

\usepackage{chngcntr}
\counterwithin{figure}{section}
\counterwithin{table}{section}
\counterwithin{equation}{section}

\usepackage[utf8]{inputenc}
\usepackage{lipsum}
\usepackage[inline, shortlabels]{enumitem}
\setlist{itemjoin ={,\enspace},itemjoin* = {\enspace}}

\usepackage{hyperref}
\hypersetup{
    linkcolor=blue,
    citecolor=blue,
}

\allowdisplaybreaks

\usetikzlibrary{calc}

\makeatletter
\newcommand*\bigcdot{\mathpalette\bigcdot@{.5}}
\newcommand*\bigcdot@[2]{\mathbin{\vcenter{\hbox{\scalebox{#2}{$\m@th#1\bullet$}}}}}
\makeatother

\newenvironment{myfont}{\fontfamily{phv}\selectfont}{\par}
\def\beg{\begin{myfont}}
\def\en{\end{myfont}}

\usetikzlibrary{positioning}

\def\tb{\textbf}

\def\t{\text}

\def\bp{\begin{proof}}
\def\ep{\end{proof}}

\def\be{\begin{enumerate}}
\def\ee{\end{enumerate}}

\def\bi{\begin{itemize}}
\def\ei{\end{itemize}}

\def\multiset#1#2{\ensuremath{\left(\kern-.3em\left(\genfrac{}{}{0pt}{}{#1}{#2}\right)\kern-.3em\right)}}

\newtheorem{lemma}{Lemma}[section]
\newtheorem{theorem}[lemma]{Theorem}
\newtheorem{corollary}[lemma]{Corollary}

\newtheorem{proposition}[lemma]{Proposition}

\theoremstyle{definition}
\newtheorem{definition}[lemma]{Definition}
\newtheorem{remark}[lemma]{Remark}

\title{Classifications of $\Gamma$-colored minuscule posets and \\ $P$-minuscule Kac--Moody representations\thanks{The main classification results of this paper appeared in a 2019 doctoral thesis \cite{Str} written under the supervision of Robert A. Proctor at the University of North Carolina.}}

\author{Michael C. Strayer \\ Hampden--Sydney College \\ Hampden--Sydney, VA 23943 U.S.A. \\ mstrayer@hsc.edu}
\date{\today}                                           

\begin{document}

\maketitle 

\begin{spacing}{1.1}

    
    
    
    
    
    
    
    
    
    



\begin{abstract}
The $\Gamma$-colored $d$-complete and $\Gamma$-colored minuscule posets unify and generalize multiple classes of colored posets introduced by R.A. Proctor, J.R. Stembridge, and R.M. Green.  In previous work, we showed that $\Gamma$-colored minuscule posets are necessary and sufficient to build from colored posets certain representations of Kac--Moody algebras that generalize minuscule representations of semisimple Lie algebras.  In this paper we classify $\Gamma$-colored minuscule posets, which also classifies the corresponding representations.  We show that $\Gamma$-colored minuscule posets are precisely disjoint unions of colored minuscule posets of Proctor and connected full heaps of Green.  Connected finite $\Gamma$-colored minuscule posets can be realized as certain posets of coroots in the corresponding finite Lie type.
\end{abstract}

\vspace{.2in} 



\vspace{-.2in}



















\end{spacing}


\begin{spacing}{1.15}

\section{Introduction}

This paper is the third in a series containing \cite{Unify} and \cite{dC-class}.  In \cite{Str,Unify}, we introduced two new axiomatic definitions of locally finite posets colored by the nodes of a Dynkin diagram.  These ``$\Gamma$-colored $d$-complete'' and ``$\Gamma$-colored minuscule'' posets correspond to particular types of representations of Kac--Moody algebras (or subalgebras) that generalize the minuscule representations of semisimple Lie algebras.  We classified all $\Gamma$-colored $d$-complete posets in \cite{dC-class} and applied this classification to $\lambda$-minuscule Weyl group elements of D. Peterson \cite{Car} and to ``upper $P$-minuscule'' representations.  In this paper, we classify all $\Gamma$-colored minuscule posets and, consequently, all corresponding ``$P$-minuscule'' representations.  For connected posets, these classifications are summarized in Table \ref{RepClass}.  This table shows that the $\Gamma$-colored minuscule and $\Gamma$-colored $d$-complete posets provide a unified axiomatic framework for colored minuscule posets of R.A. Proctor \cite{BLPP}, dominant minuscule heaps of J.R. Stembridge \cite{Ste} (which are reformulations of the colored $d$-complete posets of Proctor \cite{Wave}), and full heaps of R.M. Green \cite{Gre}.

\begin{table}[h!]
    \centering
    \begin{tabular}{|l||c|c|}
        \hline 
         \textbf{Connected posets} & \textbf{Finite} & \textbf{Infinite} \\
         \hline 
         \hline 
        \textbf{$\boldsymbol{\Gamma}$-colored minuscule} & Colored minuscule posets & Full heaps \\
        First introduced by: & R.A. Proctor (1984) & R.M. Green (2007) \\
        \hline 
        \textbf{$\boldsymbol{\Gamma}$-colored $\boldsymbol{d}$-complete} & Dominant minuscule heaps & Filters of full heaps \\
        First introduced by: & J.R. Stembridge (2001) & This author (2019) \\
        \hline
    \end{tabular}
    \caption{The classifications of connected $\Gamma$-colored minuscule and $\Gamma$-colored $d$-complete posets.}
    \label{RepClass}
\end{table}



A minuscule representation of a semisimple Lie algebra is an irreducible highest weight representation in which all weights are in the Weyl group orbit of the highest weight.  The weight diagrams of these representations are distributive lattices under the standard order on weights.  Proctor proved this fact in \cite{BLPP} and used it to introduce ``irreducible minuscule posets'' as the posets of join irreducible elements of these weight diagrams.  While initially uncolored, he then assigned a simple root to each element of the poset in Theorem 11 of that paper; we view this assignment as a coloring of the poset with the nodes of the associated Dynkin diagram.  
This coloring relies on realizing minuscule posets as certain subsets of coroots for the associated simple Lie algebra.  We review
Proctor's work from \cite{BLPP} in Section \ref{SectionHistory}, and in Theorem \ref{ThmColoredCoroots} we obtain this coroot realization in our axiomatic setting for connected finite $\Gamma$-colored minuscule posets.


Both uncolored and colored minuscule posets have appeared in numerous applications.  They were used to provide a Littlewood--Richardson rule to calculate the ($K$-theoretic) Schubert structure constants $c_{\lambda,\mu}^\nu$ of minuscule varieties (e.g. see \cite{ThYo,BuSa}).  
The rowmotion action on the order ideals of a minuscule poset exhibits the cyclic sieving phenomenon \cite{RushI}, the order ideal cardinality statistic is homomesic \cite{RushII}, and the lattice of order ideals satisfies the coincidental down-degree expectations property \cite{Hop,RushIII}; see \cite[\S 3]{Roby} for a survey of some of these topics. 
N.J. Wildberger used colored minuscule posets in Lie types $A$, $D$, and $E$ to construct simple Lie algebras of these types as well as Chevalley bases for these algebras in \cite{Wil}; see also \cite[\S 7.2]{Gre}.
Of combinatorial interest, Proctor showed that minuscule posets are Sperner and (with R. Stanley's help) Gaussian; see \cite[\S 6]{BLPP} and \cite[\S 11.3]{Gre} for the latter property.
Minuscule posets were generalized to uncolored and colored $d$-complete posets by Proctor \cite{Wave,DDCT}, and later to dominant minuscule heaps by Stembridge \cite{Ste}, to study $\lambda$-minuscule Weyl group elements.  These finite posets have also been studied extensively and used in many applications; see \cite{Pechenik,NaOk,Kokyuroku,ProScop,Ste,dC-class}.

While the original uncolored and colored minuscule and $d$-complete posets are finite, full heaps are unbounded above and below by definition.  Aside from this difference, the defining axioms for full heaps are mostly an adaptation of Stembridge's defining axioms for dominant minuscule heaps.  Green introduced these colored posets in \cite{Gre1,Gre2} and wrote a Cambridge monograph \cite{Gre} exploring many of their uses in representation theory and algebraic geometry.  The extended slant lattices of M. Hagiwara \cite{Hag2} are early appearances of full heaps colored by Dynkin diagrams of affine type $\tilde{A}$.  Our original work \cite{Str,Unify} was inspired by Green's construction of representations of affine Kac--Moody algebras using raising and lowering operators defined on the lattices of ``proper ideals'' of full heaps.  Each full heap has a unique ``principal subheap'' up to isomorphism, which is one of the colored minuscule posets of Proctor.  It appears in an infinitely repeating motif within the full heap.  These embedded principal subheaps provide a connection between Green's representations of affine Kac--Moody algebras and minuscule representations of semisimple Lie algebras, as the latter can be viewed as embedded in the former; see \cite[Prop. 5.5.5]{Gre}.

The main result of this paper is the classification of all $\Gamma$-colored minuscule posets in Theorem \ref{ThmMinusculeClassify}.  This classification handles the finite and infinite poset cases separately.  
Additionally, we split the finite case into simply laced and multiply laced cases.
To handle the finite multiply laced case, we apply Stembridge's classification \cite{Ste} of dominant minuscule heaps colored by multiply laced Dynkin diagrams.  This result was an extension of Proctor's classification \cite{DDCT} of $d$-complete posets, which exist only in the simply laced case.  
However, our approach in the simply laced case does not use Proctor's classification (see Section \ref{SectionSLTopTrees}).
In the infinite case, we apply the classification of connected full heaps by Green and Z.S. McGregor-Dorsey \cite{Gre,McD}.  In both cases, we also use several results from \cite{dC-class} developed for the more general $\Gamma$-colored $d$-complete posets; we restate these results in Section \ref{SectionDefinitions} for the reader's convenience.  Theorem \ref{ThmMinusculeClassify} and Theorem 26 of \cite{dC-class} combine to fill in and justify Table \ref{RepClass}.

Though our techniques are combinatorial, our motivations include representation theory.  The main result of \cite{Unify} stated that $\Gamma$-colored minuscule posets are necessary and sufficient to build $P$-minuscule representations of (derived) Kac--Moody algebras.  So the classification in Theorem \ref{ThmMinusculeClassify} also classifies all $P$-minuscule representations, which we state in Theorem \ref{ThmPMinusculeClassify}.  Except for a minor difference between our underlying vector space generated by the filter-ideal ``splits'' of $P$ and Green's underlying vector space generated by the proper ideals of $P$, the $P$-minuscule representations created from connected infinite $\Gamma$-colored minuscule posets are precisely the representations of affine Kac--Moody algebras produced by Green in \cite{Gre1,Gre}.  The $P$-minuscule representations created from connected finite $\Gamma$-colored minuscule posets are precisely the minuscule representations of semisimple Lie algebras.  Hence $\Gamma$-colored minuscule posets provide another framework to link Green's representations with minuscule representations.



We give definitions and preliminary results from \cite{dC-class} in Section \ref{SectionDefinitions} and dedicate Sections \ref{SectionMultiplyLaced}--\ref{SectionSimplyLaced} to classifying connected finite $\Gamma$-colored minuscule posets.  We apply the above referenced classification of Stembridge in Section \ref{SectionMultiplyLaced} to handle the multiply laced case, and we handle the simply laced case in Sections \ref{SectionSLTopTreeLemmas}--\ref{SectionSimplyLaced}.  These sections contain several results that run parallel to results obtained in the classification of uncolored $d$-complete posets of Proctor \cite{DDCT}, including obtaining the basic form of the ``top tree'' and a downward extension process to produce new $\Gamma$-colored $d$-complete posets from a given top tree.  We give the classifications of $\Gamma$-colored minuscule posets and $P$-minuscule representations in Section \ref{SectionMainResults}, and in Section \ref{SectionHistory} we realize $\Gamma$-colored minuscule posets as certain posets of coroots in the finite case.


\section{Definitions and preliminaries}\label{SectionDefinitions}


Let $P$ be a nonempty partially ordered set.  We follow \cite{Sta} for the following commonly used terms: interval, covering relations and the Hasse diagram, order ideal and order filter, saturated chain, linear extension, ranked poset and rank function, convex subposet, order dual poset, disjoint union of posets, distributive lattice, and join irreducible element of a lattice.  We also follow the definitions and notation established in \cite{dC-class}.  We will use letters such as $z,y,x,\dots$ to denote elements of $P$.  Let $x,y \in P$.  If $x$ is covered by $y$, then we write $x \to y$.  We say $x$ and $y$ are \emph{neighbors} in $P$ if $x \to y$ or $y \to x$.  If $x \le y$, then we denote the open and closed intervals between $x$ and $y$ respectively by $(x,y)$ and $[x,y]$.  We often assume $P$ is finite and always require it to be \emph{locally finite}, meaning that all intervals in $P$ are finite.  If a poset cannot be written as a disjoint union of two of its nonempty subposets, then it is \emph{connected}.  The \emph{connected components} of $P$ are the connected subposets of $P$ whose disjoint union is $P$.

Let $\Gamma$ be a finite set.  We use letters such as $a,b,c,\dots$ to denote the elements of $\Gamma$ and call them \emph{colors}.  Fix integers $\theta_{ab}$ for $a,b \in \Gamma$ that satisfy the following requirements:
\begin{enumerate}[(i),nosep]
    \item For all $a \in \Gamma$, we have $\theta_{aa} = 2$.
    \item For all distinct $a,b \in \Gamma$, we have $\theta_{ab} \le 0$ and $\theta_{ba} \le 0$. 
    \item For all distinct $a,b \in \Gamma$, we have $\theta_{ab} = 0$ if and only if $\theta_{ba} = 0$.
\end{enumerate}
We say $a$ and $b$ are \emph{distant} when $\theta_{ab} = 0$ and \emph{adjacent} when $\theta_{ab} < 0$.  If $a$ and $b$ are adjacent, we write $a \sim b$ and say $a$ is \emph{$k$-adjacent to $b$} (respectively $b$ is \emph{$l$-adjacent to $a$}) when $\theta_{ab} = -k$ (respectively $\theta_{ba} = -l$).

This choice of integers may be realized equivalently with a finite graph with finitely many nodes and no loops.  The nodes are the elements of $\Gamma$.  Let $a,b \in \Gamma$ be distinct.  If $\theta_{ab}\theta_{ba} = 0$, then there is no edge between $a$ and $b$.  If $\theta_{ab}\theta_{ba} = 1$, then there is a single undirected edge between $a$ and $b$.  If $\theta_{ab} \theta_{ba} > 1$, then there is a directed edge from $a$ to $b$ (respectively from $b$ to $a$) decorated with the integer $-\theta_{ab}$ (respectively $-\theta_{ba}$).  The resulting graph (together, possibly, with some edge decorations) is the \emph{Dynkin diagram} corresponding to the above choice of integers, and will also be denoted $\Gamma$.
The choice of integers satisfying (i)--(iii) that produces a given Dynkin diagram $\Gamma$ is unique and can be easily recovered from $\Gamma$ itself; this is our standard practice.

We say a Dynkin diagram $\Gamma$ is \emph{acyclic} if the underlying simple graph, obtained by replacing each pair of directed edges by a single undirected edge, is acyclic.  If $\Gamma$ is a simple graph (equivalently, if $\theta_{ab} \in \{-1,0,2\}$ for all $a,b \in \Gamma$), then we say $\Gamma$ is \emph{simply laced}.  Otherwise $\Gamma$ is \emph{multiply laced}.


The elements of $P$ are \emph{$\Gamma$-colored} by equipping $P$ with a surjective coloring function $\kappa : P \to \Gamma$ onto the nodes of $\Gamma$.  
By abuse of notation we typically refer to the triple $(P,\Gamma,\kappa)$ as $P$.
If $P_1$ is $\Gamma_1$-colored by $\kappa_1$ and $P_2$ is $\Gamma_2$-colored by $\kappa_2$, then $P_1$ is \emph{isomorphic} to $P_2$ if there is a poset isomorphism $\pi : P_1 \to P_2$ and a graph isomorphism $\gamma : \Gamma_1 \to \Gamma_2$ such that $\kappa_2\pi = \gamma \kappa_1$.  We write $P_1 \cong P_2$.
Unless otherwise specified, isomorphisms in this paper are $\Gamma$-colored poset isomorphisms as defined here.
Whenever we say that a $\Gamma$-colored poset with certain properties is unique, it is understood to mean unique up to such an isomorphism.

We define properties that $P$ may satisfy with respect to a $\Gamma$-coloring:
\begin{itemize}[nosep]
    \item [] (EC) Elements with equal colors are comparable.
    \item [] (NA) Neighbors have adjacent colors.
    \item [] (AC) Elements with adjacent colors are comparable.
\end{itemize}
\noindent For each color $a$, we define $P_a := \{x \in P \ | \ \kappa(x) = a\}$.  We say that $x,y \in P_a$ are \emph{consecutive elements of the color $a$} if $x < y$ in $P$ and $(x,y)$ contains no elements of the color $a$.
\begin{itemize}[nosep]
    \item [] (ICE2) For every $a \in \Gamma$, if $x < y$ are consecutive elements of the color $a$, then $\sum_{z \in (x,y)} -\theta_{\kappa(z),a} = 2$.
\end{itemize}
\noindent For an element $x \in P$, we define the set $U(x,P) := \{y \in P \ | \ y > x \ \text{and} \ \kappa(y) \sim \kappa(x)\}$.  Dually, we define the set $L(x,P) := \{y \in P \ | \ y < x \ \text{and} \ \kappa(y) \sim \kappa(x)\}$.  Let $k \ge 1$.
\begin{itemize}[nosep]
    \item [] (UCB$k$) For every $a \in \Gamma$, if $x$ is maximal in $P_a$, then $U(x,P)$ is finite and $\sum_{y \in U(x,P)} -\theta_{\kappa(y),a} \le k$.
    \item [] (LCB$k$) For every $a \in \Gamma$, if $x$ is minimal in $P_a$, then $L(x,P)$ is finite and $\sum_{y \in L(x,P)}-\theta_{\kappa(y),a} \le k$.
\end{itemize}

\noindent The properties ICE2, UCB$k$, and LCB$k$ control the ``census'' of elements of colors that are adjacent to a given color in either an interval or an upper or lower ``frontier'' of the poset.  The latter two properties are thus the \emph{frontier census properties}.  These properties were introduced in \cite{Unify} in the simply laced case and in \cite{Str,dC-class} in the multiply laced case.  We follow the property naming conventions of \cite{dC-class}.

These properties are used to give the main colored poset definitions in this paper.

\begin{definition}
A \emph{$\Gamma$-colored $d$-complete} poset is a locally finite $\Gamma$-colored poset (of any cardinality) that satisfies EC, NA, AC, ICE2, and UCB1.  If it also satisfies LCB1, then it is a \emph{$\Gamma$-colored minuscule} poset. 
\end{definition}


\noindent We remark that the order dual of a $\Gamma$-colored minuscule poset is also $\Gamma$-colored minuscule.


J.R. Stembridge introduced and classified dominant minuscule heaps in \cite{Ste}.  These finite $\Gamma$-colored posets correspond to $\lambda$-minuscule Weyl group elements for dominant integral weights $\lambda$.  We give Stembridge's definition, translated to our conventions for notation and terminology.


\begin{definition}\label{dominantminusculeheap}
A \emph{dominant minuscule heap} is a finite $\Gamma$-colored poset that satisfies
\begin{itemize}[nosep]
    \item [] (S1) All neighbors in $P$ have colors that are equal or adjacent in $\Gamma$, and the colors of incomparable elements are distant.
    \item [] (S2) For every $a \in \Gamma$, the open interval between any two consecutive elements of color $a$ either contains (i) exactly two elements whose colors are adjacent to $a$, and their colors are 1-adjacent to $a$, or (ii) exactly one element, and the color of this element is 2-adjacent to $a$.
    \item [] (S3) For every $a \in \Gamma$, an element that is maximal in $P_a$ is covered by at most one element, and this element is maximal among all elements of some color that is 1-adjacent to $a$.
    \item [] (S4) The Dynkin diagram $\Gamma$ is acyclic.
\end{itemize}
\end{definition}




\noindent Stembridge did not initially require $\kappa$ to be surjective, and his version of S4 was that the colors appearing in $\kappa(P)$ index an acyclic subdiagram of $\Gamma$.  Outside of Section \ref{SectionHistory}, we only use his classification of dominant minuscule heaps from Section 4 of \cite{Ste}, which is completed under the assumption that $\kappa$ is surjective.




Several of Stembridge's coloring axioms for dominant minuscule heaps were slightly modified by R.M. Green to define full heaps in \cite{Gre1,Gre2}.  The main difference is cardinality; while dominant minuscule heaps are finite, full heaps must satisfy
\begin{itemize}[nosep]
    \item [] (G3) For every $a \in \Gamma$, the set $P_a$ is isomorphic as an uncolored poset to $\mathbb{Z}$.
\end{itemize}
\noindent This is the third axiom from Green's definition of full heaps given in \cite{Gre}, translated to our conventions for notation and terminology.  Hence full heaps are infinite $\Gamma$-colored posets that are unbounded above and below.  In \cite[Cor. 7]{dC-class}, we showed that full heaps can be defined as follows.

\begin{definition}
A locally finite $\Gamma$-colored poset $P$ is a \emph{full heap} if it satisfies EC, NA, AC, ICE2, and G3.
\end{definition}




We typically only use the following definitions when $P$ is a finite $\Gamma$-colored $d$-complete poset.  We define $Ch(P)$ to be the set of elements $x \in P$ whose principal filter $\{y \in P \ | \ y \ge x\}$ is a chain.  
We follow Stembridge and define the \emph{top tree} $T$ of $P$ to be the set of maximal elements of each color.  
The property S3, 
which holds for finite $\Gamma$-colored $d$-complete posets by Theorem \ref{Thm9Str3} stated below, implies that $T$ is a filter of $P$.  
We usually assume $P$ is connected when considering top trees, in which case $T$ is a rooted tree.  
We note that EC and the surjectivity of $\kappa$ imply $\kappa|_T$ is a bijection.
A finite $\Gamma$-colored $d$-complete poset $P$ is \emph{slant irreducible} if it is connected and whenever $x,y \in T$ satisfy $x \to y$, the element $y$ is not the only element of its color in $P$.


We close this section by stating results appearing in \cite{dC-class} concerning $\Gamma$-colored $d$-complete posets.

\begin{theorem}[Theorem 9 of \cite{dC-class}]\label{Thm9Str3}
Let $P$ be a finite $\Gamma$-colored poset.  Then $P$ is a $\Gamma$-colored $d$-complete poset if and only if it is a dominant minuscule heap. $\qed$
\end{theorem}

\noindent Because of Theorem \ref{Thm9Str3}, we freely use the defining axioms for dominant minuscule heaps when working with finite $\Gamma$-colored $d$-complete posets.

\begin{theorem}[Proposition 13 of \cite{dC-class}]\label{ThmProp13Str3}
Suppose $P$ is a $\Gamma$-colored $d$-complete poset and for every $b \in \Gamma$, the set $P_b$ is bounded above.  Then $P$ is finite. $\qed$
\end{theorem}

\begin{theorem}[Theorem 22 of \cite{dC-class}]\label{Thm22Str3}
Let $P$ be a connected infinite $\Gamma$-colored poset.  Then $P$ is a $\Gamma$-colored $d$-complete poset if and only if it is a filter of some connected full heap. $\qed$
\end{theorem}


\section{Connected finite \texorpdfstring{$\Gamma$}{Gamma}-colored minuscule posets when \texorpdfstring{$\Gamma$}{Gamma} is multiply laced}\label{SectionMultiplyLaced}

We begin the classification of connected finite $\Gamma$-colored minuscule posets in this section.  This classification will be broken into cases by Proposition \ref{PropChainorSlantIrreducible}.  We handle the multiply laced case in Theorem \ref{ThmMultLacedClassify} by applying Stembridge's classification of dominant minuscule heaps colored by multiply laced Dynkin diagrams.

\begin{remark}\label{PconniffGamma}
Corollary 24 of \cite{dC-class} shows that if $P$ satisfies EC, NA, and AC, then $P$ is connected if and only if $\Gamma$ is connected.  When $P$ satisfies these properties and we need either $P$ or $\Gamma$ to be connected, we will assume $P$ is connected and use this result without comment.
\end{remark}

We begin with two straightforward lemmas needed for Proposition \ref{PropChainorSlantIrreducible}.  

\begin{lemma}\label{LemTconnected}
Let $P$ be a connected finite $\Gamma$-colored $d$-complete poset.  Then the top tree $T$ is connected.
\end{lemma}

\begin{proof}
Suppose for a contradiction that $T$ is the disjoint union of subposets $T_1$ and $T_2$.  Then $\kappa(T_1)$ and $\kappa(T_2)$ are nonempty, and they partition $\Gamma$ since $\kappa|_T$ is a bijection.  Since $\Gamma$ is connected, there are colors $a \in \kappa(T_1)$ and $b \in \kappa(T_2)$ with $a \sim b$.  Let $x \in T_1$ and $y \in T_2$ be the respective elements in $T$ of colors $a$ and $b$.  Then $x$ and $y$ are comparable by AC; without loss of generality, assume $x < y$ in $P$.  Then $x < y$ in $T$ as well.  This contradiction shows $T$ is connected.
\end{proof} 

Our second lemma has appeared in other forms; for example, see Proposition F2 of \cite{DDCT}.

\begin{lemma}\label{LemUniqueMaxElt}
Let $P$ be a connected finite $\Gamma$-colored $d$-complete poset.  Then $P$ has a unique maximal element.  If $P$ is $\Gamma$-colored minuscule, then it also has a unique minimal element.
\end{lemma}

\begin{proof}
Suppose for a contradiction that $x$ and $y$ are distinct maximal elements of $P$.  Then $x$ and $y$ are in the top tree $T$.  Since $T$ is connected by Lemma \ref{LemTconnected}, choose a path from $x$ to $y$ in the Hasse diagram of $T$.  This path must start by moving down from $x$ and end by moving up to $y$.  Thus there must be an element $z \in T$ along this path that is covered by two elements.  This violates S3, so $P$ has a unique maximal element.



Now suppose $P$ is $\Gamma$-colored minuscule.  The order dual poset $P^*$ is also $\Gamma$-colored minuscule and hence $\Gamma$-colored $d$-complete.  Thus it has a unique maximal element, which is the unique minimal element of $P$.
\end{proof}

Our next result divides the classification of connected finite $\Gamma$-colored minuscule posets into cases.

\begin{proposition}\label{PropChainorSlantIrreducible}
Let $P$ be a connected finite $\Gamma$-colored minuscule poset.
Then either $P$ is a chain and $\Gamma$ is simply laced, or $P$ is slant irreducible as a $\Gamma$-colored $d$-complete poset.
\end{proposition}


\begin{proof}
Suppose $P$ is not slant irreducible as a $\Gamma$-colored $d$-complete poset and let $T$ be the top tree of $P$.  Choose a pair of neighbors $x \to y$ in $T$ such that $y$ is the only element of its color in $P$ and $x$ is minimal in the set of all pairs of neighbors satisfying this condition.  Note that $y$ is in the top tree $T^*$ of the order dual $P^*$ as well.  Applying S3 to $P^*$, we see $x \in T^*$.  Since $x \in T \cap T^*$, we see that $x$ is the only element of its color in $P$.

Let $a := \kappa(x)$ and $b := \kappa(y)$.  By NA we have $a \sim b$.  Suppose that there is some $u \in P$ with $u \to x$, and set $c := \kappa(u)$.  The minimality of the choice of $x$ shows $u \notin T$.  Thus there is some $v \in P_c$ with $u < v$.  By NA we see that $a \sim c$, and by AC we see that $x < v$.  Since $x$ is maximal in $P_a$, by UCB1 we know there is at most one element above $x$ with color adjacent to $a$.  This shows $v = y$, and hence $c = b$.  This contradicts that $y$ is the only element of its color in $P$.  Hence $x$ is minimal in $P$.  Since $x$ is the unique minimal element of $P$ by Lemma \ref{LemUniqueMaxElt}, the filter generated by $x$ is $P$.  By potentially repeated applications of S3 moving upward from $x$, we see $P$ is a chain and $P = T$.

Let $e,f \in \Gamma$ be adjacent.  Since $P = T$, let $s$ and $t$ be the respective unique elements of colors $e$ and $f$ in $P$.  Since $P$ is a chain, without loss of generality assume $s < t$.  Then UCB1 (respectively LCB1) shows $f$ is 1-adjacent to $e$ (respectively $e$ is 1-adjacent to $f$).  Hence $\Gamma$ is simply laced.
\end{proof}


Suppose $P$ is a connected finite $\Gamma$-colored minuscule poset and $\Gamma$ is multiply laced.  The preceding proposition shows $P$ is slant irreducible as a $\Gamma$-colored $d$-complete poset.  Since $P$ is a dominant minuscule heap by Theorem \ref{Thm9Str3}, we can apply Stembridge's classification of slant irreducible dominant minuscule heaps colored by multiply laced Dynkin diagrams.  For convenience, we restate that result here using our conventions for notation and terminology.

\begin{theorem}[Theorem 4.2 of \cite{Ste}]\label{ThmStembridgeClassification}
Let $P$ be a slant irreducible $\Gamma$-colored $d$-complete poset.
Then either $\Gamma$ is simply laced or it has the form
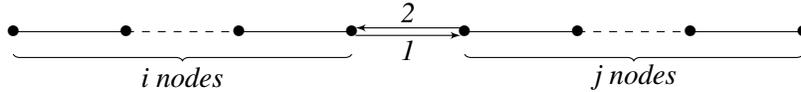
\begin{center} 
\begin{tikzpicture}
        \tikzset{edge/.style = {->,> = latex'}}
        \node at (-3,0){$\bullet$};
        \node at (-1.5,0){$\bullet$};
        \node at (0,0){$\bullet$};
        \node at (1.5,0){$\bullet$};
        \node at (4.5,0){$\bullet$};
        \node at (6,0){$\bullet$};
        \node at (7.5,0){$\bullet$};
        \node at (3,0){$\bullet$};
        \draw (-3,0) -- (-1.5,0) (0,0) -- (1.5,0);
        \draw[dashed] (-1.5,0) -- (0,0);
        \draw (3,0) -- (4.5,0) (6,0) -- (7.5,0);
        \draw[dashed] (4.5,0) -- (6,0);
        \draw[edge] (2.95,.05)  to (1.55,0.05);
       \draw[edge] (1.55,-.05) to (2.95,-0.05);
       \node at (2.25,.25){2};
       \node at (2.25,-.25){1};
       
       \draw [
    decoration={
        brace, mirror
    },
    decorate
] (-3,-.3) -- (1.5,-.3);
\draw [
    decoration={
        brace, mirror
    },
    decorate
] (3,-.3) -- (7.5,-.3);
    \node at (-.75,-.6){$i$ nodes};
    \node at (5.25,-.6){$j$ nodes};
        \end{tikzpicture}
\end{center} 
for some $i,j \ge 1$.  If $\Gamma$ is multiply laced, then either
\begin{enumerate}[(a),nosep]
    \item We have $i = 1$ and $j \ge 1$ and $P$ is isomorphic to an order filter of a poset of the form displayed in Figure \ref{FigSlantIrredDomMinHeaps}(a) containing at least one element of each color, or 
    \item We have $i > 1$ and $j \ge 1$ and $P$ is isomorphic to a poset of the form displayed in Figure \ref{FigSlantIrredDomMinHeaps}(b).
\end{enumerate}
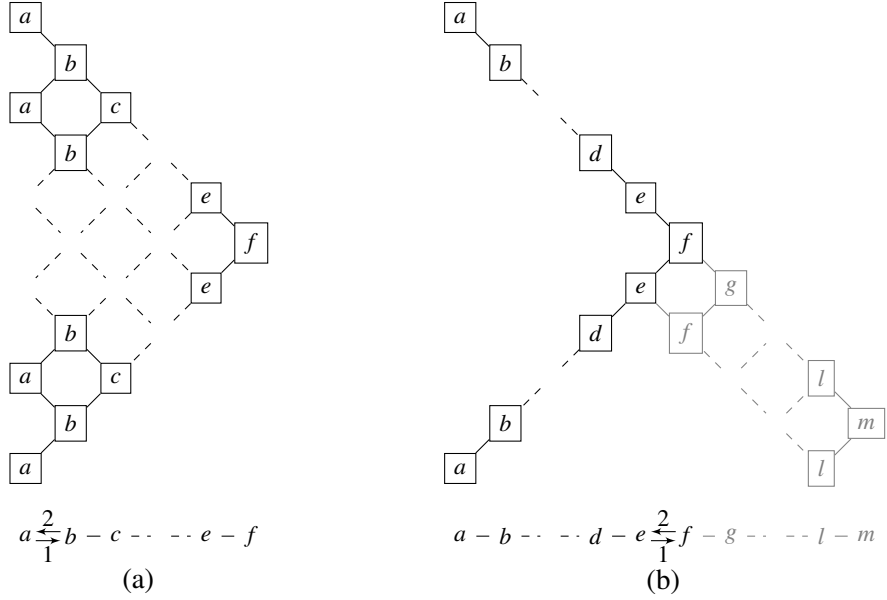
\begin{figure}[h!]
\begin{center} 
\begin{minipage}{0.35\textwidth} 
    \begin{tikzpicture}[scale=.6]
    \tikzset{edge/.style = {->,> = latex'}}
        \node[draw] (Z) at (0,0){{\footnotesize $a$}};
        \node[draw] (Y) at (1,-1){{\footnotesize $b$}};
        \node[draw] (X) at (2,-2){{\footnotesize $c$}};
        \node[draw] (W) at (4,-4){{\footnotesize $e$}};
        \node[draw] (V) at (5,-5){{\footnotesize $f$}};
        
        \node[draw] (U) at (0,-2){{\footnotesize $a$}};
        \node[draw] (T) at (1,-3){{\footnotesize $b$}};
        \node[draw] (S) at (4,-6){{\footnotesize $e$}};
        
        \node[draw] (R) at (1,-7){{\footnotesize $b$}};
        \node[draw] (Q) at (2,-8){{\footnotesize $c$}};
        
        \node[draw] (P) at (0,-8){{\footnotesize $a$}};
        \node[draw] (O) at (1,-9){{\footnotesize $b$}};
        \node[draw] (N) at (0,-10){{\footnotesize $a$}};
        
        \node (1) at (3,-3){};
        \node (2) at (2,-4){};
        \node (3) at (3,-5){};
        \node (4) at (0,-4){};
        \node (5) at (1,-5){};
        \node (6) at (2,-6){};
        \node (7) at (3,-7){};
        \node (8) at (0,-6){};
        
        \draw (Z) -- (Y) -- (X) -- (T) -- (U) -- (Y);
        \draw (W) -- (V) -- (S);
        \draw (N) -- (O) -- (Q) -- (R) -- (P) -- (O);
        \draw[dashed] (X) -- (1) -- (W);
        \draw[dashed] (T) -- (2) -- (3) -- (S);
        \draw[dashed] (4) -- (5) -- (6) -- (7);
        \draw[dashed] (8) -- (R);
        \draw[dashed] (T) -- (4);
        \draw[dashed] (1) -- (2) -- (5) -- (8);
        \draw[dashed] (W) -- (3) -- (6) -- (R);
        \draw[dashed] (S) -- (7) -- (Q);
        
        \node (A) at (0,-11.5){{\footnotesize $a$}};
        \node (B) at (1,-11.5){{\footnotesize $b$}};
        \node (C) at (2,-11.5){{\footnotesize $c$}};
        \node (E) at (4,-11.5){{\footnotesize $e$}};
        \node (F) at (5,-11.5){{\footnotesize $f$}};
        \node (9) at (3,-11.5){};
        
        \node (AT) at (0,-11.4){};
        \node (BT) at (1,-11.4){};
        \node (AB) at (0,-11.6){};
        \node (BB) at (1,-11.6){};
        
        \draw (B) -- (C) (E) -- (F);
        \draw[dashed] (C) -- (9) -- (E);
        \draw[edge] (BT) to (AT);
        \draw[edge] (AB) to (BB);
        
        \node at (.5,-11.1){{\footnotesize $2$}};
        \node at (.5,-11.9){{\footnotesize $1$}};
        
        \node at (2.5,-12.5){(a)};
    \end{tikzpicture}
    \end{minipage}%
\begin{minipage}{0.35\textwidth}
    \begin{tikzpicture}[scale=.6]
    \tikzset{edge/.style = {->,> = latex'}}
        \node[draw] (Z) at (0,0){{\footnotesize $a$}};
        \node[draw] (Y) at (1,-1){{\footnotesize $b$}}; 
        \node (3) at (2,-2){};
        \node[draw] (X) at (3,-3){{\footnotesize $d$}};
        \node[draw] (W) at (4,-4){{\footnotesize $e$}};
        \node[draw] (V) at (5,-5){{\footnotesize $f$}};
        \node [draw, gray] (U) at (6,-6){{\footnotesize $g$}};
        \node (4) at (7,-7){};
        \node [draw, gray] (T) at (8,-8){{\footnotesize $l$}};
        \node [draw, gray] (S) at (9,-9){{\footnotesize $m$}};
        \node[draw] (R) at (4,-6){{\footnotesize $e$}};
        \node [draw, gray] (Q) at (5,-7){{\footnotesize $f$}};
        \node (P) at (6,-8){};
        \node (5) at (7,-9){};
        \node [draw, gray] (O) at (8,-10){{\footnotesize $l$}};
        \node[draw] (N) at (3,-7){{\footnotesize $d$}};
        \node (6) at (2,-8){};
        \node[draw] (M) at (1,-9){{\footnotesize $b$}};
        \node[draw] (L) at (0,-10){{\footnotesize $a$}};
        
        \node (A) at (0,-11.5){{\footnotesize $a$}};
        \node (B) at (1,-11.5){{\footnotesize $b$}};
        \node (1) at (2,-11.5){};
        \node (D) at (3,-11.5){{\footnotesize $d$}};
        \node (E) at (4,-11.5){{\footnotesize $e$}};
        \node (F) at (5,-11.5){{\footnotesize $f$}};
        \node [gray] (G) at (6,-11.5){{\footnotesize $g$}};
        \node (2) at (7,-11.5){};
        \node [gray] (I) at (8,-11.5){{\footnotesize $l$}};
        \node [gray] (J) at (9,-11.5){{\footnotesize $m$}};
        
        \node at (4.5,-12.5){(b)};
        
        \node (ET) at (4,-11.4){};
        \node (EB) at (4,-11.6){};
        \node (FT) at (5,-11.4){};
        \node (FB) at (5,-11.6){};
        
        \draw (Z) -- (Y) (X) -- (W) -- (V);
        \draw [gray] (T) -- (S) -- (O);
        \draw [gray] (Q) -- (U);
        \draw [gray] (U)-- (V);
        \draw (V) -- (R);
        \draw [gray] (R) -- (Q);
        \draw (R) -- (N) (M) -- (L);
        \draw[dashed] (Y) -- (3) -- (X);
        \draw [gray,dashed] (U) -- (4) -- (T) (P) -- (5) -- (O);
        \draw [dashed] (N) -- (6) -- (M);
        \draw (A) -- (B) (D) -- (E);
        \draw[gray] (F) -- (G) (I) -- (J);
        \draw[dashed] (B) -- (1) -- (D);
        \draw[gray,dashed](G) -- (2) -- (I);
        \draw[dashed,gray] (P) -- (4) (5) -- (T);
        \draw [dashed, gray](P) -- (Q);
        
        \draw [edge] (FT) to (ET);
        \draw [edge] (EB) to (FB);
        
        \node at (4.5,-11.1){{\footnotesize 2}};
        \node at (4.5,-11.9){{\footnotesize 1}};
    \end{tikzpicture}
\end{minipage}
    \caption{The two families of slant irreducible $\Gamma$-colored $d$-complete posets from Theorem \ref{ThmStembridgeClassification}.  In both families, vertically aligned elements have the same color.  The first family includes all order filters of the displayed poset containing at least one element of each color.  (a) When $i = 1$ and $j \ge 1$, a $\Gamma$-colored minuscule poset of type $B$.  (b) When $i > 1$ and $j = 1$, a $\Gamma$-colored minuscule poset of type $C$.}
    \label{FigSlantIrredDomMinHeaps}
\end{center}
\end{figure}
\end{theorem}

We use this result to classify the connected finite $\Gamma$-colored minuscule posets when $\Gamma$ is multiply laced.  To state this classification concisely, we name the types of $\Gamma$-colored posets appearing in it.

\begin{definition}\label{DefTypeBC}
Let $P$ be a finite $\Gamma$-colored poset.  Refer to the integers $i,j \ge 1$ used in Theorem \ref{ThmStembridgeClassification}.
\begin{enumerate}[(a),nosep]
    \item If $P$ is isomorphic to the poset displayed in Figure \ref{FigSlantIrredDomMinHeaps}(a) for $i = 1$ and some $j \ge 1$, then $P$ has \emph{type $B$}.
    \item If $P$ is isomorphic to the poset displayed in Figure \ref{FigSlantIrredDomMinHeaps}(b) for some $i > 1$ and $j = 1$, then $P$ has \emph{type $C$}.
\end{enumerate}
\end{definition}

\noindent It is routine to check that $\Gamma$-colored posets of types $B$ and $C$ are $\Gamma$-colored minuscule.  We have required $i > 1$ in the type $C$ definition to prevent overlap between types $B$ and $C$.  Requiring $j = 1$ for posets of type $C$ results in the absence of gray nodes and edges in Figure \ref{FigSlantIrredDomMinHeaps}(b).  So posets of type $C$ are chains.

\begin{theorem}\label{ThmMultLacedClassify}
Let $P$ be a connected finite $\Gamma$-colored minuscule poset and assume $\Gamma$ is multiply laced.  Then $P$ has type $B$ or type $C$.
\end{theorem}

\begin{proof}
We know from Proposition \ref{PropChainorSlantIrreducible} that $P$ must be slant irreducible as a $\Gamma$-colored $d$-complete poset, so we apply Theorem \ref{ThmStembridgeClassification}.  Referring to that result, first suppose that $i = 1$.  The value of $j$ then determines the Dynkin diagram $\Gamma$ and a poset $Q$ as displayed in Figure \ref{FigSlantIrredDomMinHeaps}(a).  We know by Theorem \ref{ThmStembridgeClassification} that $P$ is isomorphic to a filter of $Q$ containing at least one element of each color.  By Lemma \ref{LemUniqueMaxElt} we know $P$ has a unique minimal element.  If the minimal element of $P$ does not have color $a$ (as displayed in Figure \ref{FigSlantIrredDomMinHeaps}(a)), then the minimal element in $P_a$ covers an element of color $b$ in $P$, violating LCB1.  Hence the unique minimal element of $P$ must have color $a$.  The only element of color $a$ in $Q$ that is less than an element of every other color in $\Gamma$ is the minimal element of $Q$.  Thus $P \cong Q$ and so $P$ has type $B$.  

Now suppose that $i > 1$.  The value of $j$ determines the Dynkin diagram $\Gamma$ and poset $P$ isomorphic to the poset displayed in Figure \ref{FigSlantIrredDomMinHeaps}(b).  Since $P$ must have a unique minimal element by Lemma \ref{LemUniqueMaxElt}, we must have $j = 1$.  Thus $P$ has type $C$.
\end{proof}


\section{Simply laced top tree lemmas and \texorpdfstring{$\Gamma$}{Gamma}-colored minuscule chains}\label{SectionSLTopTreeLemmas}

Having classified the connected finite $\Gamma$-colored minuscule posets colored by multiply laced $\Gamma$ in Section \ref{SectionMultiplyLaced}, we turn to the simply laced case.  We first establish results regarding top trees of such posets that will be used here and in later sections.  In Theorem \ref{ThmChainandSimplyLaced}, we obtain the classification of $\Gamma$-colored minuscule posets that are chains when $\Gamma$ is simply laced, completing one case given in Proposition \ref{PropChainorSlantIrreducible}.

\begin{lemma}\label{LemSimplyColored}
Let $P$ be a finite $\Gamma$-colored $d$-complete poset and assume $\Gamma$ is simply laced.  Let $x < y$ in $P$ and suppose that $[x,y]$ is a chain.  Then there are no repeated colors in this interval.
\end{lemma}

\begin{proof}
Suppose for a contradiction that there are repeated colors in $[x,y]$.  Choose $u < v$ in $[x,y]$ such that $\kappa(u) = \kappa(v)$ and there are no other repeated colors in $[u,v]$.  Since $\Gamma$ is simply laced, the interval $[u,v]$ must contain at least four elements by ICE2.  By NA, moving up in the interval $[u,v]$ corresponds to a path in $\Gamma$.  This path must be a cycle since it contains at least three distinct colors starting and ending at the same color.  Hence this path is a cycle in $\Gamma$, contradicting S4.
\end{proof}


The next lemma shows that in the finite simply laced case, the top tree $T$ is equal to the set $Ch(P)$ of all elements whose principal filter is a chain.  This is not true for posets of type $B$ or $C$.  The set $Ch(P)$ is one version of how top trees are defined for the $d$-complete posets of Proctor (e.g. \cite{Kokyuroku,ProScop}).

\begin{lemma}\label{LemTequalsC}
Let $P$ be a finite $\Gamma$-colored $d$-complete poset with top tree $T$ and assume $\Gamma$ is simply laced.  Then $T = Ch(P)$.
\end{lemma}

\begin{proof}
First suppose $x \in T$.  Then S3 shows $x$ is covered by at most one element, and if this element exists it is in $T$.  By repeating this reasoning, it follows that the filter generated by $x$ is a chain.  Hence $x \in Ch(P)$.  


Now suppose $x \in Ch(P)$.  Let $y \ge x$ be a maximal element of $P$.  Since $x \in Ch(P)$, the interval $[x,y]$ is the filter generated by $x$ and is a chain.  It has no repeated colors by Lemma \ref{LemSimplyColored}.  Hence we have $x \in T$.
\end{proof}

Let $P$ be a finite poset and let $T$ be a convex subposet of $P$.  We define $Gr(T)$ to be the simple graph whose vertices are the elements of $T$ and whose edges correspond to the covering relations in $T$.  In other words, the graph $Gr(T)$ is just the Hasse diagram of $T$ without its partial order.  We typically use $Gr(T)$ when $T$ is the top tree of a finite $\Gamma$-colored $d$-complete poset colored by simply laced $\Gamma$.  In this case, the next result shows that $Gr(T) \cong \Gamma$ as graphs, so $T$ is essentially a copy of $\Gamma$ embedded in $P$.  Knowing this, Proctor originally defined $d$-complete posets without reference to externally present Dynkin diagrams.



\begin{proposition}\label{PropTisomtoGamma}
Let $P$ be a finite $\Gamma$-colored $d$-complete poset with top tree $T$ and assume $\Gamma$ is simply laced.  Then $\tilde{\kappa} : Gr(T) \to \Gamma$ defined by
$\tilde{\kappa} := \kappa|_T$ is a graph isomorphism.
\end{proposition}

\begin{proof}
We know $\tilde{\kappa}$ is a bijection.
Suppose $x$ and $y$ are adjacent in $Gr(T)$ and thus neighbors in $T$.  By NA, we see that $\tilde{\kappa}(x)$ and $\tilde{\kappa}(y)$ are adjacent in $\Gamma$.

Now suppose colors $a$ and $b$ are adjacent in $\Gamma$.  Let $u,v \in Gr(T)$ be such that $\tilde{\kappa}(u) = a$ and $\tilde{\kappa}(v) = b$.  By AC, we know that $u$ and $v$ are comparable in $P$, so assume without loss of generality that $u < v$.  Since $u \in Ch(P)$ by Lemma \ref{LemTequalsC}, we know that $[u,v]$ is a chain.  By NA, moving up in this chain induces a path in $\Gamma$ starting at $a$ and ending at $b$.  There are no repeated colors in $[u,v]$ by Lemma \ref{LemSimplyColored}.  Since $a \sim b$, this path produces a cycle in $\Gamma$ unless $u \to v$.  Since $\Gamma$ is acyclic by S4, this shows $u$ and $v$ are adjacent in $Gr(T)$.
\end{proof}


We obtain an important result to establish the uniqueness of certain $\Gamma$-colored $d$-complete posets.


\begin{proposition}\label{PropSLIsom}
Let $P$ be a finite poset for which $Ch(P) = P$.  Then $P$ can be realized uniquely as a $\Gamma$-colored $d$-complete poset colored by simply laced $\Gamma$.
\end{proposition}



\begin{proof}
Define $\Gamma := Gr(P)$  and equip $P$ with the identity coloring $\kappa : P \to \Gamma$ on nodes.  Since there is only one element of each color, the properties EC and ICE2 are satisfied vacuously.  The properties NA and AC are satisfied by construction.  For any $x \in P$, the sum $\sum_{y \in U(x,P)} -\theta_{\kappa(y),\kappa(x)}$ is the number of elements covering $x$.  Since $Ch(P) = P$, this sum is at most one for every element of $P$.  Thus UCB1 holds, which proves that $P$ is $\Gamma$-colored $d$-complete.



Now suppose $P$ is $\Gamma'$-colored $d$-complete for simply laced Dynkin diagram $\Gamma'$ and coloring $\kappa' : P \to \Gamma'$.  By Lemma \ref{LemTequalsC}, the top tree of $P$ with respect to the coloring $\kappa'$ is $Ch(P) = P$.  Note that the maps $\text{id}_{P}$, $\text{id}_{Gr(P)}$, and $\kappa$ are all identity on the elements of $P$.  Hence we obtain the commutative diagram

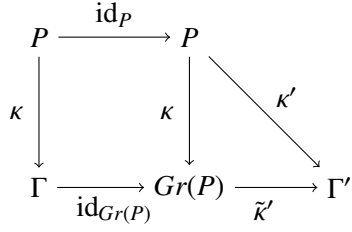
\begin{center}
\begin{tikzpicture}
    \node(A) at (0,0){$P$};
    \node(B) at (2,0){$P$};
    \node(C) at (0,-2){$\Gamma$};
    \node(D) at (2,-2){$Gr(P)$};
    \node(E) at (4,-2){$\Gamma'$};
    \node at (-.3,-1){$\kappa$};
    \node at (1.7,-1){$\kappa$};
    \node at (1,0.3){$\text{id}_{P}$};
    \node at (1,-2.3){$\text{id}_{Gr(P)}$};
    \node at (3,-2.3){$\tilde{\kappa}'$};
    \node at (3.3,-0.8){$\kappa'$};
    \draw[->] (A) -- (B);
    \draw[->] (B) -- (D);
    \draw[->] (D) -- (E);
    \draw[->] (A) -- (C);
    \draw[->] (C) -- (D);
    \draw[->] (B) -- (E);
\end{tikzpicture}
\end{center}




\noindent where $\tilde{\kappa}' : Gr(P) \to \Gamma'$ is 
the graph isomorphism of Proposition \ref{PropTisomtoGamma}.  Hence $\tilde{\kappa}' \circ \text{id}_{Gr(P)} : \Gamma \to \Gamma'$ is a graph isomorphism.  Clearly $\text{id}_{P}$ is a poset isomorphism, so the triples $(P,\Gamma,\kappa)$ and $(P,\Gamma',\kappa')$ are isomorphic.
\end{proof}

We can now classify connected finite $\Gamma$-colored minuscule posets that are chains when $\Gamma$ is simply laced.  To state this classification concisely, we first name the type of $\Gamma$-colored posets appearing in it.

\begin{definition}
Let $P$ be a finite $\Gamma$-colored poset.  If $P$ is isomorphic to a poset of the form displayed in Figure \ref{FigTypeAStandard}, then $P$ has \emph{type $A$ standard}.
\end{definition}

It is routine to check that a $\Gamma$-colored poset of type $A$ standard is $\Gamma$-colored minuscule.

\begin{theorem}\label{ThmChainandSimplyLaced}
Let $P$ be a connected finite $\Gamma$-colored minuscule poset and assume $\Gamma$ is simply laced.  If $P$ is a chain, then $P$ has type $A$ standard.
\end{theorem} 


\begin{proof}
The type $A$ standard poset of cardinality $|P|$ is a chain and is $\Gamma$-colored minuscule (and hence $\Gamma$-colored $d$-complete).  Proposition \ref{PropSLIsom} shows it must be isomorphic to $P$.
\end{proof}

\begin{figure}[h!]
    \centering
    \begin{tikzpicture}[scale=.65]
        \node (A) at (0,0){{\footnotesize $a$}};
        \node (B) at (1,0){{\footnotesize $b$}};
        \node (1) at (2,0){};
        \node (D) at (3,0){{\footnotesize $d$}};
        \node (E) at (4,0){{\footnotesize $e$}};
        
        \draw (A) -- (B) (D) -- (E);
        \draw[dashed] (B) -- (1) -- (D);
        
        \node[draw] (Z) at (0,5.5){{\footnotesize $a$}};
        \node[draw] (Y) at (1,4.5){{\footnotesize $b$}};
        \node (2) at (2,3.5){};
        \node[draw] (X) at (3,2.5){{\footnotesize $d$}};
        \node[draw] (W) at (4,1.5){{\footnotesize $e$}};
        
        \draw (Z) -- (Y) (X) -- (W);
        \draw[dashed] (Y) -- (2) -- (X);
    \end{tikzpicture} 
    \caption{A $\Gamma$-colored minuscule poset of type $A$ standard.}
    \label{FigTypeAStandard}
\end{figure}
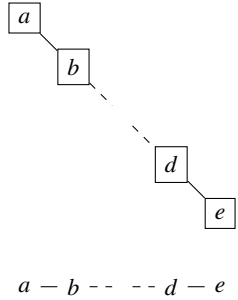

\begin{remark}\label{RemMoreThanOneColor}
Let $P$ be a connected $\Gamma$-colored minuscule poset.  If $P$ consists of a single element, then $\Gamma$ consists of one color since $\kappa$ is surjective.  If $P$ contains more than one element, then $\Gamma$ also contains more than one color; just take neighbors in $P$ (which exist since $P$ is connected) and note that they must have different colors by NA.  Hence $|P| = 1$ if and only if $|\Gamma| = 1$.  The case where $P$ consists of a single element is handled by Theorem \ref{ThmChainandSimplyLaced} and is the only case that satisfies both possibilities of Proposition \ref{PropChainorSlantIrreducible}.
\end{remark}


\section{Possible slant irreducible top trees when \texorpdfstring{$\Gamma$}{Gamma} is simply laced}\label{SectionSLTopTrees}

Proposition \ref{PropChainorSlantIrreducible} showed that a connected finite $\Gamma$-colored minuscule poset must be a chain colored by simply laced $\Gamma$ or be slant irreducible as a $\Gamma$-colored $d$-complete poset.  We handled the former case in Theorem \ref{ThmChainandSimplyLaced} and the multiply laced version of the latter case in Theorem \ref{ThmMultLacedClassify}.  The remaining case is when $P$ is slant irreducible as a $\Gamma$-colored $d$-complete poset and $\Gamma$ is simply laced.  The next three sections are dedicated to this case, which culminates in Theorem \ref{ThmSimpLacedSlantIrred}.  Our main result in this section is Corollary \ref{CorTHasShapeY}, which identifies the structure of the top tree $T$ in this case.

Dominant minuscule heaps (and hence finite $\Gamma$-colored $d$-complete posets) are reformulations of the colored $d$-complete posets of Proctor \cite{Wave}, which exist only for simply laced Dynkin diagrams.  Proctor classified the uncolored versions of these posets in \cite{DDCT} using notions of top tree and slant irreducibility defined for uncolored posets.  When these posets are colored uniquely to become colored $d$-complete posets as in \cite[Prop. 8.6]{Wave}, Proctor's definitions are equivalent to the versions used here.

In Section \ref{SectionMultiplyLaced} of this paper, we applied Stembridge's classification of slant irreducible dominant minuscule heaps to classify the connected finite $\Gamma$-colored minuscule posets colored by multiply laced $\Gamma$.  It is possible to emulate that approach and use Proctor's classification of slant irreducible uncolored $d$-complete posets to classify connected finite $\Gamma$-colored minuscule posets colored by simply laced $\Gamma$; indeed, that was our approach in \cite{Str}.  However, in addition to creating the need to pass between uncolored and colored versions of these posets, that approach relies on the equivalence between Proctor's colored $d$-complete posets and finite $\Gamma$-colored $d$-complete posets (or equivalently, Stembridge's dominant minuscule heaps) colored by simply laced Dynkin diagrams.  While the equivalence between Proctor's and Stembridge's posets is known to experts, it has not appeared in any published paper and so we avoid using it.

Some of the results of this section run parallel to results obtained by Proctor in the uncolored $d$-complete setting (e.g. \cite[\S5]{DDCT}).  Given Remark \ref{RemMoreThanOneColor}, we assume that $T$ has more than one element, or equivalently, that $\Gamma$ has more than one color.  Our first three lemmas give basic facts about $T$.

\begin{lemma}\label{LemTnotaChain}
Let $P$ be a slant irreducible $\Gamma$-colored $d$-complete poset and assume $\Gamma$ is simply laced.  If the top tree $T$ contains more than one element, then $T$ is not a chain.
\end{lemma}


\begin{proof}
Since $T$ is connected by Lemma \ref{LemTconnected} and $|T| > 1$, choose neighbors $x \to y$ in $T$.  Since $P$ is slant irreducible, the element $y$ is not the only element of its color in $P$.  Thus we see that $P \ne T$.  Choose an element $z$ that is maximal in $P - T$.  Since $T = Ch(P)$ by Lemma \ref{LemTequalsC}, the filter generated by $z$ is not a chain and thus contains incomparable elements.  Since $z$ is maximal in $P - T$, these elements are in $T$.  Hence $T$ is not a chain.
\end{proof}

\begin{lemma}\label{LemCoverTwo}
Let $P$ be a slant irreducible $\Gamma$-colored $d$-complete poset with top tree $T$.  Every element of $T$ that is not minimal in $T$ covers at most two elements in $P$.
\end{lemma}

\begin{proof}
Suppose $y \in T$ is not minimal in $T$.  By slant irreducibility, we know that $y$ is not the only element of its color in $P$.  Let $x < y$ be consecutive elements of color $\kappa(y)$.  By NA, any elements covered by $y$ in $P$ have colors adjacent to $\kappa(y)$.  By AC, these elements are in $(x,y)$.  Then ICE2 shows there are at most two such elements.
\end{proof}

Since a finite $\Gamma$-colored poset is $\Gamma$-colored $d$-complete if and only if it is a dominant minuscule heap by Theorem \ref{Thm9Str3}, we make use of Lemma 4.1 of \cite{Ste}.  We present that lemma here for convenience, translating to our conventions for notation and terminology.

\begin{lemma}[Lemma 4.1 of \cite{Ste}]\label{LemJRS4.1}
Let $P$ be a slant irreducible $\Gamma$-colored $d$-complete poset with top tree $T$.  If $s \in T$ covers two elements of $P$, then every element $y < s$ in $T$ covers an element outside of $T$.
\end{lemma}

Now we obtain the key result regarding the structure of $T$.

\begin{proposition}\label{PropUniqueSplittingElt}
Let $P$ be a slant irreducible $\Gamma$-colored $d$-complete poset and assume $\Gamma$ is simply laced.  If the top tree $T$ contains more than one element, then there exists a unique element in $T$ that covers more than one element in $T$.  This element covers exactly two elements in $T$ and is comparable to every element in $T$.
\end{proposition}

\begin{proof}
Since $T = Ch(P)$ by Lemma \ref{LemTequalsC}, every element of $T$ is covered by at most one element.  We know $T$ is connected by Lemma \ref{LemTconnected}.  We also know by Lemma \ref{LemTnotaChain} that $T$ is not a chain.  Hence some element $s \in T$ covers more than one element in $T$.  Then Lemma \ref{LemCoverTwo} shows $s$ covers exactly two elements in $T$.

Suppose for a contradiction that $w \in T$ is incomparable to $s$.  
Again using that every element of $T$ is covered by at most one element, a path from $s$ to $w$ in the Hasse diagram of $T$ must pass through some element $u > s$ that covers two elements in $T$.
Then $s$ must cover an element outside of $T$ by Lemma \ref{LemJRS4.1}.  Thus $s$ covers three or more elements in $P$, violating Lemma \ref{LemCoverTwo}.  So every element in $T$ is comparable to $s$.

If an element greater than $s$ covers two elements of $T$, then Lemma \ref{LemJRS4.1} shows $s$ covers an element outside of $T$.  Similarly, if an element less than $s$ covers two elements of $T$, then Lemma \ref{LemJRS4.1} shows this element must also cover an element outside of $T$.  Both situations violate Lemma \ref{LemCoverTwo}.  Thus $s$ is the unique element that covers more than one element in $T$.
\end{proof}

Let $P$ be a slant irreducible $\Gamma$-colored $d$-complete poset and assume $\Gamma$ is simply laced.  Suppose that the top tree $T$ contains more than one element.  Let $s \in T$ be the unique element covering two elements in $T$ guaranteed by Proposition \ref{PropUniqueSplittingElt}
and call $s$ the \emph{splitting element} of $T$.  
Then $s$ is comparable to every element in $T$.  Let $i \ge 1$ be the number of elements in the principal filter generated by $s$, which is a chain since $T = Ch(P)$.  Since only $s$ can cover more than one element in $T$, the two elements covered by $s$ are each maximal elements of disjoint saturated chains in $T$ below $s$.  Suppose these chains consist of $j \ge 1$ and $k \ge 1$ elements in $T$, respectively.  Following notation used in \cite{DDCT}, we say that $T$ has \emph{shape $Y(i;j,k)$}.  We also write $T = Y(i;j,k)$.  Our convention is to require $k \ge j \ge 1$ and for the branch containing $j$ (respectively $k$) elements to appear on the left (respectively right) side of the Hasse diagram of $T$ or $P$, as in Figure \ref{FigY}.  This discussion obtains the following result.



\begin{corollary}\label{CorTHasShapeY}
Let $P$ be a slant irreducible $\Gamma$-colored $d$-complete poset and assume $\Gamma$ is simply laced.  If the top tree $T$ contains more than one element, then $T = Y(i;j,k)$ for some integers $i \ge 1$ and $k \ge j \ge 1$. $\qed$
\end{corollary}

\begin{figure}[h!]
    \centering
    \begin{tikzpicture}[scale=.65]
        \node (a) at (0,0){$\bullet$};
        \node (b) at (1,-1){$\bullet$};
        \node (c) at (2,-2){};
        \node (d) at (3,-3){$\bullet$};
        \node (e) at (2,-4){$\bullet$};
        \node (f) at (1,-5){$\bullet$};
        \node (g) at (4,-4){$\bullet$};
        \node (h) at (5,-5){};
        \node (i) at (6,-6){$\bullet$};
        \draw (a) -- (b) (e) -- (d) -- (g);
        \draw[dashed] (b) -- (c) -- (d) (e) -- (f) (g) -- (h) -- (i);
        \node at (3,-1.25){{\footnotesize $i$ elements}};
        \node at (0,-4.25){{\footnotesize $j$ elements}};
        \node at (6.6,-4.75){{\footnotesize $k$ elements}};
        \node at (5.25,-3){{\footnotesize Splitting element $s$}};
        
        \draw [
    decoration={
        brace
    },
    decorate
] (0,0.25) -- (3.25,-3);

\draw [
    decoration={
        brace, mirror
    },
    decorate
] (2,-3.75) -- (.75,-5);

\draw [
    decoration={
        brace
    },
    decorate
] (4,-3.75) -- (6.25,-6);
        
    \end{tikzpicture}
    \caption{A top tree $T$ of shape $Y(i;j,k)$ for some integers $i \ge 1$ and $k \ge j \ge 1$.}
    \label{FigY}
\end{figure}
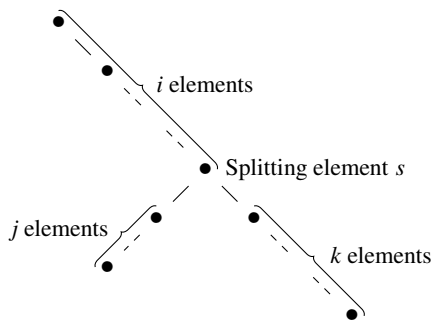



\begin{remark}\label{RemDoubleDip}
Let $i \ge 1$ and $k \ge j \ge 1$ and let $T = Y(i;j,k)$.  
By Proposition \ref{PropSLIsom} the poset $T$ can be realized uniquely as a $\Gamma$-colored $d$-complete poset colored by simply laced $\Gamma$.
This will be used in the next two sections, where we start with a $\Gamma$-colored $d$-complete poset $T = Y(i;j,k)$ and examine its potential to be the top tree of a $\Gamma$-colored minuscule poset for some simply laced $\Gamma$.  Note that we must require $T$ to be $\Gamma$-colored $d$-complete since filters of $\Gamma$-colored minuscule posets are $\Gamma$-colored $d$-complete.
\end{remark}

\section{Extending top trees \texorpdfstring{$T = Y(i;j,k)$}{T=Y(i;j,k)} to \texorpdfstring{$\Gamma$}{Gamma}-colored minuscule posets}\label{SectionExtending}

In this section, we describe a process that starts with a $\Gamma$-colored $d$-complete poset $T = Y(i;j,k)$ for some $i \ge 1$ and $k \ge j \ge 1$ and either produces the unique $\Gamma$-colored minuscule poset with top tree $T$, or shows that no such poset exists.  This is the outcome of Theorem \ref{ThmAlgorithmExtension}, which is our main result in this section.  This process occurs through a downward extension similar to the process used by Proctor \cite{DDCT} to classify the uncolored $d$-complete posets.  We also used a similar downward extension process to show connected infinite $\Gamma$-colored $d$-complete posets are order filters of full heaps \cite[\S 5]{dC-class}.



\begin{definition}\label{DefExtendableBya}
Let $P$ be a connected finite $\Gamma$-colored $d$-complete poset and let $a \in \Gamma$.  We say $P$ is \emph{extendable by $a$} if there is a connected finite $\Gamma$-colored $d$-complete poset $P'$ such that 
\begin{enumerate}[(i),nosep]
\item The poset $P$ is a filter of $P'$, 
\item The difference $P' - P$ consists of a single element, which we call the \emph{extending element}, and 
\item The color of the extending element is $a$.  
\end{enumerate} 
We also say $P'$ is an \emph{extension of $P$} by $a$ and write $P' \xrightarrow[]{a} P$.
\end{definition}

Our first lemma shows that the color extensions of Definition \ref{DefExtendableBya} are uniquely determined by the colored structure of the poset being extended.  Recall for $y \in P$ that $L(y,P) = \{x \in P \ | \ x < y \ \t{and} \ \kappa(x) \sim \kappa(y)\}$.

\begin{lemma}\label{LemExtensionUnique}
Let $P$ be a connected finite $\Gamma$-colored $d$-complete poset and let $a \in \Gamma$.  Suppose $P'$ is an extension of $P$ by $a$ with extending element $x$.  Then $x \to u$ in $P'$ if and only if $u$ is minimal in $L(y,P)$, where $y$ is the minimal element of $P_a$.
\end{lemma}

\begin{proof}
The element $x$ is minimal in $P'$.  So by AC applied to $P'$, the set $L(y,P)$ consists of the elements in $(x,y)$ whose colors are adjacent to $a$.  By NA and AC, every element that covers $x$ in $P'$ has color adjacent to $a$ and is in $(x,y)$.  So $x \to u$ in $P'$ if and only if $u$ is minimal among elements in $(x,y)$ whose colors are adjacent to $a$; that is, if and only if $u$ is minimal in $L(y,P)$.
\end{proof}




\noindent We build on this result to show that when colors are distant, the extension order does not matter.

\begin{lemma}\label{LemExtensionUniquePt2}
Let $P$ be a connected finite $\Gamma$-colored $d$-complete poset and let $b,c \in \Gamma$ be distant.  If there are extensions $P'' \xrightarrow[]{c} P' \xrightarrow[]{b} P$ and $Q'' \xrightarrow[]{b} Q' \xrightarrow[]{c} P$, then $P'' \cong Q''$.
\end{lemma}

\begin{proof}
Let $x$ and $y$ be the extending elements of colors $b$ and $c$ in $P' - P$ and $P'' - P'$, respectively.  
By NA and the fact that $b$ and $c$ are distant, we see $x$ and $y$ are not neighbors.
Thus they are each only covered by elements in $P$ and are minimal in $P''$.
Hence $P'' - \{x\}$ is a filter of $P''$ and is therefore $\Gamma$-colored $d$-complete.
This shows $P$ can first be extended by $c$ to $P'' - \{x\}$ with extending element $y$ and then by $b$ to $P''$ with extending element $x$.  Lemma \ref{LemExtensionUnique} shows these extensions are respectively isomorphic to $Q'$ and $Q''$.
\end{proof}

Our main tool for determining when a color extension exists will be the lower frontier census, defined next.  We used a slightly more general version of this tool in Section 4 of \cite{dC-class}.

\begin{definition}
Let $P$ be a $\Gamma$-colored poset that satisfies EC and let $b \in \Gamma$.  If $y$ is minimal in $P_b$ and $L(y,P)$ is finite, then define
$$ L_b(P) := \sum_{x \in L(y,P)} -\theta_{\kappa(x),b} $$
and call it the \emph{lower frontier census} of $P$ for the color $b$.
\end{definition}

We give a necessary and sufficient condition for a color extension to exist. 

\begin{lemma}\label{LemExtensionCondition}
Let $P$ be a connected finite $\Gamma$-colored $d$-complete poset and let $a \in \Gamma$.  Then $P$ is extendable by $a$ if and only if $L_a(P) = 2$.
\end{lemma}


\begin{proof}
First suppose that $P$ is extendable by $a$ and suppose that $P' \xrightarrow[]{a} P$ has extending element $x$.  Let $y$ be minimal in $P_a$.  Then $x < y$ are consecutive elements of the color $a$ in $P'$.  
By AC applied to $P'$, the set $L(y,P)$ consists of the elements in $(x,y)$ whose colors are adjacent to $a$.
Hence by ICE2 applied to $P'$ we have $L_a(P) = \sum_{z \in L(y,P)} -\theta_{\kappa(z),a} = \sum_{z \in (x,y)} -\theta_{\kappa(z),a} = 2$.

Now suppose that $L_a(P) = 2$; we produce the required poset $P'$.  
Create a new element $x$ with color $a$ and set $P' := P \cup \{x\}$.  
Define the order on $P'$ to be the reflexive transitive closure of the covering relations in $P$ along with the covering relation(s) $x \to u$ if and only if $u$ is minimal in $L(y,P)$, where $y$ is minimal in $P_a$.  The properties EC, NA, AC, and UCB1 are immediate in $P'$ by construction since these properties hold in $P$.  Note that $x < y$ is the only occurrence of consecutive elements of the same color not contained in $P$, and so this is the only instance we must check to verify ICE2 for $P'$.  Also by construction, the elements in $(x,y)$ with colors adjacent to $a$ are precisely the elements in $L(y,P)$, and so $\sum_{z \in (x,y)} -\theta_{\kappa(z),a} = L_a(P) = 2$.  Hence ICE2 holds and $P'$ is $\Gamma$-colored $d$-complete.
\end{proof}

Lemmas \ref{LemExtensionUnique}, \ref{LemExtensionUniquePt2}, and \ref{LemExtensionCondition} focus on when and how the extension process described below will proceed.  The next result adds three scenarios which would cause it to terminate.

\begin{lemma}\label{LemExtensionFail}
Let $P$ be a connected finite $\Gamma$-colored $d$-complete poset.
\begin{enumerate}[(a),nosep]
    \item The poset $P$ is $\Gamma$-colored minuscule if and only if $L_b(P) \le 1$ for every $b \in \Gamma$.
    \item The poset $P$ is not a filter of any $\Gamma$-colored minuscule poset if there exists $b \in \Gamma$ with $L_b(P) > 2$.
    \item The poset $P$ is not a filter of any $\Gamma$-colored minuscule poset if there exist adjacent $b,c \in \Gamma$ with $L_b(P) = L_c(P) = 2$.
\end{enumerate}
\end{lemma}

\begin{proof}
Since $P$ is finite, the condition $L_b(P) \le 1$ for every $b \in \Gamma$ is a reformulation of LCB1.  Thus (a) holds.




Now suppose there exists a color $b \in \Gamma$ with $L_b(P) > 2$.  Suppose for a contradiction that $P$ is a filter of some $\Gamma$-colored minuscule poset $Q$.  By Theorem \ref{ThmProp13Str3} the poset $Q$ is finite.  The set $Q - P$ must contain an element of color $b$ since $Q$ satisfies LCB1.  Let $x$ be the maximal element of color $b$ in $Q - P$.  Let $y$ be minimal in $P_b$.  Then $x < y$ are consecutive elements of color $b$ in $Q$.  Every element in $L(y,P)$ must be in $(x,y)$ by AC.  Since $L_b(P) > 2$, this violates ICE2 for $Q$.  Thus (b) holds.  

Now suppose there exist adjacent $b,c \in \Gamma$ with $L_b(P) = L_c(P) = 2$.  Suppose again for a contradiction that $P$ is a filter of some $\Gamma$-colored minuscule poset $Q$.  Repeating the reasoning from the above paragraph, we can find consecutive elements $s < t$ of color $b$ and $u < v$ of color $c$, with $s,u \in Q - P$ and $t,v \in P$.  These elements are pairwise comparable by EC and AC.  If $s < u$, then $(s,t)$ violates ICE2.  If $u < s$, then $(u,v)$ violates ICE2.  Thus (c) holds.
\end{proof}

The final lemma in this section provides the starting point for the downward extension process.

\begin{lemma}\label{LemExtensionSeedUnique}
Fix integers $i \ge 1$ and $k \ge j \ge 1$ and let $T$ be the unique $\Gamma$-colored $d$-complete poset of shape $Y(i;j,k)$ colored by simply laced $\Gamma$.  Then $L_b(T) \le 2$ for all $b \in \Gamma$ and $L_c(T) = 2$ if and only if $c = \kappa(s)$, where $s$ is the splitting element of $T$.
\end{lemma}

\begin{proof}
The poset $T$ is its own top tree.  Using Proposition \ref{PropTisomtoGamma}, the sum $L_b(T)$ for any color $b \in \Gamma$ is just the number of elements covered by the element of color $b$ in $T$.  The result follows since $s$ covers two elements in $T$ and every other element of $T$ covers at most one element.
\end{proof}


Fix integers $i \ge 1$ and $k \ge j \ge 1$ and let $T$ be the unique $\Gamma$-colored $d$-complete poset of shape $Y(i;j,k)$ colored by simply laced $\Gamma$.  We now describe the extension process used to show that either $T$ is not the top tree of any $\Gamma$-colored minuscule poset, or that there is a unique $\Gamma$-colored minuscule poset with top tree $T$.  Each stage in this process has two steps.  Here we merely list the steps; we justify why each step may be performed in the proof of Theorem \ref{ThmAlgorithmExtension}.  Define $P_0 := T$.

\begin{enumerate}[nosep]
    \item \textit{Assessment step:} Start with a valid $\Gamma$-colored $d$-complete poset $P_l$ for some $l \ge 0$.  This poset is either given (if $l = 0$) or produced in an earlier stage of the process (if $l > 0$).  Assess the poset $P_l$ and determine if the process terminates, which can happen in one (or more) of three ways:
    \begin{enumerate}[(a),nosep] 
    \item The poset $P_l$ satisfies $L_b(P_l) \le 1$ for every $b \in \Gamma$.
    \item The poset $P_l$ satisfies $L_b(P_l) > 2$ for some $b \in \Gamma$.
    \item The poset $P_l$ satisfies $L_b(P_l) = L_c(P_l) = 2$ for some pair of adjacent $b,c \in \Gamma$.
    \end{enumerate}
    Otherwise, the process continues and we define the \emph{color extension set} $E_{l+1}(T) := \{b \in \Gamma \ | \ L_b(P_l) = 2\}$.
    \item \textit{Extension step:} Fix an ordering $a_1,\dots,a_r$ of the colors in $E_{l+1}(T)$ and form the sequence of extensions $P^{(r)} \xrightarrow[]{a_r} P^{(r-1)} \xrightarrow[]{a_{r-1}} \cdots \xrightarrow[]{a_2} P^{(1)} \xrightarrow[]{a_1} P$.  Define $P_{l+1} := P^{(r)}$ and return to the assessment step.
\end{enumerate}


\noindent We say the posets $P_l$ for $l > 0$ produced in this process are \emph{rank complete} and justify that terminology in the following proof.  This process must terminate by Theorem \ref{ThmProp13Str3}.  We say the final poset produced in this process is a \emph{maximal rank complete} $\Gamma$-colored $d$-complete poset with top tree $T$.


\begin{theorem}\label{ThmAlgorithmExtension}
Fix integers $i \ge 1$ and $k \ge j \ge 1$ and let $T$ be the unique $\Gamma$-colored $d$-complete poset of shape $Y(i;j,k)$ colored by simply laced $\Gamma$.
Then there is a unique maximal rank complete $\Gamma$-colored $d$-complete poset $P$ with top tree $T$, and
\begin{enumerate}[(a),nosep]
\item If $L_b(P) \le 1$ for all $b \in \Gamma$, then $P$ is the unique $\Gamma$-colored minuscule poset with top tree $T$.  
\item If $L_b(P) > 1$ for some $b \in \Gamma$, then there is no $\Gamma$-colored minuscule poset with top tree $T$.
\end{enumerate}
\end{theorem}




\begin{proof}
It follows from \cite[Prop. 3.1(b) and Cor. 3.4]{Ste} that finite $\Gamma$-colored posets satisfying S1, S2, and S4 are ranked.  By Theorem \ref{Thm9Str3}, this shows that all finite $\Gamma$-colored $d$-complete posets are ranked.  Define the unique maximal element of $T$ to have rank $-i$, and choose the rank function such that if $x \to y$ in $T$ then the rank of $x$ is one \textit{greater} than the rank of $y$.  For example, the rank of the splitting element $s$ is $-1$.  This particular rank function is chosen for notational reasons that will be convenient for the proof, and it uniquely determines the rank of any element in a connected finite $\Gamma$-colored $d$-complete poset with top tree $T$.





The poset $P_0 = T$ never terminates at the assessment step by Lemma \ref{LemExtensionSeedUnique}.  Moreover, by that lemma we have $E_1(T) = \{c\}$, where $c$ is the color of $s$.  Thus the extension step consists of a single extension by $c$ to produce $P_1$.  Lemma \ref{LemExtensionUnique} shows the extension $P_1$ is unique, with the extending element in $P_1 - P_0$ covered precisely by the elements below $s$ with colors adjacent to $c$.  By NA, these are the two elements of rank 0 covered by $s$.  Hence the extending element in $P_1 - P_0$ has rank $1$ in $P_1$.







Now let $l > 0$ and suppose that there is a unique $\Gamma$-colored $d$-complete poset $P_l$ produced by this process, and that the elements in $P_l - P_{l-1}$ have rank $l$ in $P_l$.  Suppose the process does not terminate at the assessment step for $P_l$ and define the color extension set $E_{l+1}(T)$.  The colors in $E_{l+1}(T)$ are pairwise distant since the process did not terminate via condition (c) at the assessment step for $P_l$.  Hence extending $P_l$ by some $b \in E_{l+1}(T)$ does not affect the lower frontier census for any other $a \in E_{l+1}(T)$.  This implies by Lemma \ref{LemExtensionCondition} that the extension step may proceed for $P_l$, and in any order.  Lemma \ref{LemExtensionUniquePt2} implies that the resulting poset $P_{l+1}$ is the unique $\Gamma$-colored $d$-complete poset produced by the extension step applied to $P_l$.
No color in $E_{l+1}(T)$ was in $E_l(T)$, since $L_b(P_l) = 0$ for all $b \in E_l(T)$.  
So colors appeared in $E_{l+1}(T)$ as a result of the extension step that produced $P_l$.  Hence for every color in $E_{l+1}(T)$, there is at least one element in $P_l - P_{l-1}$ with an adjacent color.  Since the elements of $P_l - P_{l-1}$ are minimal in $P_l$, this shows by Lemma \ref{LemExtensionUnique} that each element in $P_{l+1} - P_l$ is covered by at least one element in $P_l - P_{l-1}$.  Since the latter elements have rank $l$, the elements in $P_{l+1} - P_l$ each have rank $l + 1$ in $P_{l+1}$.


Thus iterating this process produces unique rank complete $\Gamma$-colored $d$-complete posets until it terminates.  Therefore the maximal rank complete $\Gamma$-colored $d$-complete poset $P$ with top tree $T$ is unique.





Now suppose $Q$ is a $\Gamma$-colored minuscule poset with top tree $T$.  By Theorem \ref{ThmProp13Str3} the poset $Q$ is finite.  Since $T$ is a filter of $Q$ and $Q$ is ranked, it follows that $Q$ can be reconstructed from $T$ by filling in elements downwardly rank-by-rank.  Since every filter of $Q$ must be $\Gamma$-colored $d$-complete, any such rank-by-rank extension of $T$ to $Q$ must proceed via the color extensions of this section.  Lemma \ref{LemExtensionCondition} shows extension opportunities for a given color arise precisely when the lower frontier census of that color equals two.  
Since $Q$ satisfies LCB1, we see every possible extension must occur to successfully produce $Q$.  In other words, we may not proceed with extending by elements of the next rank until all possible extensions are completed with elements of the current rank; no extensions may be skipped.
Thus the process used to produce the unique maximal rank complete $\Gamma$-colored $d$-complete poset $P$ is the beginning of this downward rank-by-rank extension process to produce $Q$ from $T$, and so $P$ is a filter of $Q$.  This shows $P$ must be isomorphic to a filter of every $\Gamma$-colored minuscule poset with top tree $T$. 


Now we consider how the extension process terminated to produce $P$.  
If $L_b(P) \le 1$ for every $b \in \Gamma$, then Lemma \ref{LemExtensionFail}(a) shows that $P$ is $\Gamma$-colored minuscule and Lemma \ref{LemExtensionCondition} shows that $P$ cannot be extended by any color.
This proves (a).  Now suppose that $L_b(P) > 1$ for some $b \in \Gamma$, so that the process terminated at the assessment step condition (b) or (c).  Then Lemma \ref{LemExtensionFail}(b) or (c), respectively, shows that $P$ is not a filter of any $\Gamma$-colored minuscule poset.  This proves (b).
\end{proof}




\section{Finite \texorpdfstring{$\Gamma$}{Gamma}-colored minuscule posets with top tree \texorpdfstring{$T = Y(i;j,k)$}{T=Y(i;j,k)}}\label{SectionSimplyLaced}

Theorem \ref{ThmSimpLacedSlantIrred} is our main result in this section, where we classify the slant irreducible $\Gamma$-colored minuscule posets for which $\Gamma$ is simply laced and contains more than one color.  We use this result with Proposition \ref{PropChainorSlantIrreducible} and Theorems \ref{ThmMultLacedClassify} and \ref{ThmChainandSimplyLaced} to handle the connected finite $\Gamma$-colored minuscule case in Theorem \ref{ThmMinusculeClassify}.



Our main tool for obtaining Theorem \ref{ThmSimpLacedSlantIrred} is Theorem \ref{ThmAlgorithmExtension}.  Part (a) of that result will be used in Proposition \ref{PropNewSL} to produce connected finite $\Gamma$-colored minuscule posets.  Part (b) will be used in Proposition \ref{PropNewSLFail} to rule out all other cases.




\begin{proposition}\label{PropNewSL}
Fix integers $i \ge 1$ and $k \ge j \ge 1$ and let $T$ be the unique $\Gamma$-colored $d$-complete poset of shape $Y(i;j,k)$ colored by simply laced $\Gamma$.  Then there exists a unique connected finite $\Gamma$-colored minuscule poset with top tree $T$ in each of the following cases:
\begin{enumerate}[(a),nosep]
    \item When $i = 1$ and $k \ge j \ge 1$, as displayed in Figure \ref{FigMinSL}(a),
    \item When $i > 1$ and $k = j = 1$, as displayed in Figure \ref{FigMinSL}(b),
    \item When $i = 2$ and $j = 1$ and $k > 1$, as displayed in Figure \ref{FigMinSL}(c) separately for even $k$ and odd $k$,
    \item When $i = 3$ and $j = 1$ and $k = 2$, as displayed in Figure \ref{FigMinSL}(d), and
    \item When $i = 4$ and $j = 1$ and $k = 2$, as displayed in Figure \ref{FigMinSL}(e).
\end{enumerate}
\end{proposition}


\begin{proof}
Each poset displayed in Figure \ref{FigMinSL} can be obtained from the downward extension process of Section \ref{SectionExtending} starting with the specified top tree $T = Y(i;j,k)$.
It is easily verified that each is $\Gamma$-colored minuscule.  Theorem \ref{ThmAlgorithmExtension}(a) shows they are the unique $\Gamma$-colored minuscule posets with their respective top trees.
\end{proof}

\begin{figure}[h !]
    \centering
        \begin{tikzpicture}
            \node at (0,0){\begin{tikzpicture}[scale=.475]
        \node (A) at (0,-3){{\tiny $a$}};
        \node (B) at (1,-3){{\tiny $b$}};
        \node (1) at (2,-3){};
        \node (D) at (3,-3){{\tiny $d$}};
        \node (E) at (4,-3){{\tiny $e$}};
        \node (F) at (5,-3){{\tiny $f$}};
        \node (2) at (6,-3){};
        \node (G) at (7,-3){{\tiny $g$}};
        \node (H) at (8,-3){{\tiny $h$}};
        \node (I) at (9,-3){{\tiny $i$}};
        \node (3) at (10,-3){};
        \node (L) at (11,-3){{\tiny $l$}};
        \node (M) at (12,-3){{\tiny $m$}};
        
        
        \draw (A) -- (B) (D) -- (E) -- (F) (G) -- (H) -- (I) (L) -- (M);
        \draw[dashed] (B) -- (1) -- (D) (F) -- (2) -- (G) (I) -- (3) -- (L);
        
        \node[draw] (E1) at (4,10.5){{\tiny $e$}};
        \node[draw] (F1) at (5,9.5){{\tiny $f$}};
        \node (21) at (6,8.5){};
        \node (G1) at (7,7.5){};
        \node (H1) at (8,6.5){};
        \node (I1) at (9,5.5){};
        \node (31) at (10,4.5){};
        \node[draw] (L1) at (11,3.5){{\tiny $l$}};
        \node[draw] (M1) at (12,2.5){{\tiny $m$}};
        
        \node[draw] (D1) at (3,9.5){{\tiny $d$}};
        \node[draw] (E2) at (4,8.5){{\tiny $e$}};
        \node (F2) at (5,7.5){};
        \node (22) at (6,6.5){};
        \node (G2) at (7,5.5){};
        \node (H2) at (8,4.5){};
        \node (I2) at (9,3.5){};
        \node (32) at (10,2.5){};
        \node[draw] (L2) at (11,1.5){{\tiny $l$}};
        
        \node (11) at (2,8.5){};
        \node (D2) at (3,7.5){};
        \node (E3) at (4,6.5){};
        \node (F3) at (5,5.5){};
        \node (23) at (6,4.5){};
        \node (G3) at (7,3.5){};
        \node (H3) at (8,2.5){};
        \node (I3) at (9,1.5){};
        \node (33) at (10,0.5){};
        
        \node[draw] (B1) at (1,7.5){{\tiny $b$}};
        \node (12) at (2,6.5){};
        \node (D3) at (3,5.5){};
        \node (E4) at (4,4.5){};
        \node (F4) at (5,3.5){};
        \node (24) at (6,2.5){};
        \node (G4) at (7,1.5){};
        \node[draw] (H4) at (8,0.5){{\tiny $h$}};
        \node[draw] (I4) at (9,-0.5){{\tiny $i$}};
        
        \node[draw] (A1) at (0,6.5){{\tiny $a$}};
        \node[draw] (B2) at (1,5.5){{\tiny $b$}};
        \node (13) at (2,4.5){};
        \node (D4) at (3,3.5){};
        \node (E5) at (4,2.5){};
        \node (F5) at (5,1.5){};
        \node (25) at (6,0.5){};
        \node[draw] (G5) at (7,-0.5){{\tiny $g$}};
        \node[draw] (H5) at (8,-1.5){{\tiny $h$}};

        \node (9) at (2,8.5){};
        \node (9) at (2,8.5){};
        
        \draw [
    decoration={
        brace, mirror
    },
    decorate
] (0,-3.3) -- (3,-3.3);
\draw [
    decoration={
        brace, mirror
    },
    decorate
] (5,-3.3) -- (12,-3.3);
\node at (1.5,-3.85){{\footnotesize $j$ nodes}};
\node at (8.5,-3.85){{\footnotesize $k$ nodes}};

    \draw (E1) -- (F1) -- (E2) -- (D1) -- (E1) (B1) -- (A1) -- (B2) (G5) -- (H5) -- (I4) -- (H4) -- (G5) (L1) -- (M1) -- (L2);
    
    \draw[dashed] (F1) -- (21) -- (G1) -- (H1) -- (I1) -- (31) -- (L1) (E2) -- (F2) -- (22) -- (G2) -- (H2) -- (I2) -- (32) -- (L2) (11) -- (D2) -- (E3) -- (F3) -- (23) -- (G3) -- (H3) -- (I3) -- (33) (B1) -- (12) -- (D3) -- (E4) -- (F4) -- (24) -- (G4) -- (H4) (B2) -- (13) -- (D4) -- (E5) -- (F5) -- (25) -- (G5);
    \draw[dashed] (D1) -- (11) -- (B1) (E2) -- (D2) -- (12) -- (B2) (21) -- (F2) -- (E3) -- (D3) -- (13) (G1) -- (22) -- (F3) -- (E4) -- (D4) (H1) -- (G2) -- (23) -- (F4) -- (E5) (I1) -- (H2) -- (G3) -- (24) -- (F5) (31) -- (I2) -- (H3) -- (G4) -- (25) (L1) -- (32) -- (I3) -- (H4) (L2) -- (33) -- (I4);
    \end{tikzpicture}};
    
    \node at (6,0){\begin{tikzpicture}[scale=.475]
        \node (A) at (0,0){{\tiny $a$}};
        \node (B) at (1,0){{\tiny $b$}};
        \node (1) at (2,0){};
        \node (D) at (3,0){{\tiny $d$}};
        \node (E) at (4,0){{\tiny $e$}};
        \node (F) at (5,0.5){{\tiny $f$}};
        \node (G) at (5,-0.5){{\tiny $g$}};
        
        \draw (A) -- (B) (D) -- (E) (F) -- (E) -- (G);
        \draw [dashed] (B) -- (1) -- (D);
        
        \node[draw] (A1) at (0,12){{\tiny $a$}};
        \node[draw] (B1) at (1,11){{\tiny $b$}};
        \node (11) at (2,10){};
        \node[draw] (D1) at (3,9){{\tiny $d$}};
        \node[draw] (E1) at (4,8){{\tiny $e$}};
        \node[draw] (F1) at (3,7){{\tiny $f$}};
        \node[draw] (G1) at (5,7){{\tiny $g$}};
        \node[draw] (E2) at (4,6){{\tiny $e$}};
        \node[draw] (D2) at (3,5){{\tiny $d$}};
        \node (12) at (2,4){};
        \node[draw] (B2) at (1,3){{\tiny $b$}};
        \node[draw] (A2) at (0,2){{\tiny $a$}};
        \draw [dashed] (B1) -- (11) -- (D1) (D2) -- (12) -- (B2);
        \draw (A1) -- (B1) (D1) -- (E1) -- (F1) -- (E2) -- (D2) (E1) -- (G1) -- (E2) (B2) -- (A2);
        
        \draw [
    decoration={
        brace, mirror
    },
    decorate
] (0,-0.3) -- (4,-0.3);
        \node at (2,-0.85){{\footnotesize $i$ nodes}};
    \end{tikzpicture}};
    
    \node at (0,-8){\begin{tikzpicture}[scale=.475]
        \node[draw] (Z) at (0,0){{\tiny $a$}};
        \node[draw] (Y) at (1,-1){{\tiny $b$}};
        \node[draw] (X) at (2,-2){{\tiny $d$}};
        \node[draw] (W) at (4,-4){{\tiny $f$}};
        \node[draw] (V) at (5,-5){{\tiny $g$}};
        
        \node[draw] (U) at (0,-2){{\tiny $c$}};
        \node[draw] (T) at (1,-3){{\tiny $b$}};
        \node[draw] (S) at (4,-6){{\tiny $f$}};
        
        \node[draw] (R) at (1,-7){{\tiny $b$}};
        \node[draw] (Q) at (2,-8){{\tiny $d$}};
        
        \node[draw] (P) at (0,-8){{\tiny $a$}};
        \node[draw] (O) at (1,-9){{\tiny $b$}};
        \node[draw] (N) at (0,-10){{\tiny $c$}};
        
        \node (1) at (3,-3){};
        \node (2) at (2,-4){};
        \node (3) at (3,-5){};
        \node (4) at (0,-4){};
        \node (5) at (1,-5){};
        \node (6) at (2,-6){};
        \node (7) at (3,-7){};
        \node (8) at (0,-6){};
        
        \draw (Z) -- (Y) -- (X) -- (T) -- (U) -- (Y);
        \draw (W) -- (V) -- (S);
        \draw (N) -- (O) -- (Q) -- (R) -- (P) -- (O);
        \draw[dashed] (X) -- (1) -- (W);
        \draw[dashed] (T) -- (2) -- (3) -- (S);
        \draw[dashed] (4) -- (5) -- (6) -- (7);
        \draw[dashed] (8) -- (R);
        \draw[dashed] (T) -- (4);
        \draw[dashed] (1) -- (2) -- (5) -- (8);
        \draw[dashed] (W) -- (3) -- (6) -- (R);
        \draw[dashed] (S) -- (7) -- (Q);
        
        \node (A) at (0,-11.5){{\tiny $a$}};
        \node (B) at (1,-12){{\tiny $b$}};
        \node (C) at (0,-12.5){{\tiny $c$}};
        \node (D) at (2,-12){{\tiny $d$}};
        \node (F) at (4,-12){{\tiny $f$}};
        \node (G) at (5,-12){{\tiny $g$}};
        \node (9) at (3,-12){};
        
        
        \draw (B) -- (D) (G) -- (F) (A) -- (B) -- (C);
        \draw[dashed] (D) -- (9) -- (F);
        
        \draw [
    decoration={
        brace, mirror
    },
    decorate
] (2,-12.3) -- (5,-12.3);
        \node at (3.5,-12.85){{\footnotesize $k$ nodes (even)}};
        
        
        
    \end{tikzpicture}};
    
    \node at (6,-8){\begin{tikzpicture}[scale=.475]
        \node[draw] (Z) at (0,0){{\tiny $a$}};
        \node[draw] (Y) at (1,-1){{\tiny $b$}};
        \node[draw] (X) at (2,-2){{\tiny $d$}};
        \node[draw] (W) at (4,-4){{\tiny $f$}};
        \node[draw] (V) at (5,-5){{\tiny $g$}};
        
        \node[draw] (U) at (0,-2){{\tiny $c$}};
        \node[draw] (T) at (1,-3){{\tiny $b$}};
        \node[draw] (S) at (4,-6){{\tiny $f$}};
        
        \node[draw] (R) at (1,-7){{\tiny $b$}};
        \node[draw] (Q) at (2,-8){{\tiny $d$}};
        
        \node[draw] (P) at (0,-8){{\tiny $c$}};
        \node[draw] (O) at (1,-9){{\tiny $b$}};
        \node[draw] (N) at (0,-10){{\tiny $a$}};
        
        \node (1) at (3,-3){};
        \node (2) at (2,-4){};
        \node (3) at (3,-5){};
        \node (4) at (0,-4){};
        \node (5) at (1,-5){};
        \node (6) at (2,-6){};
        \node (7) at (3,-7){};
        \node (8) at (0,-6){};
        
        \draw (Z) -- (Y) -- (X) -- (T) -- (U) -- (Y);
        \draw (W) -- (V) -- (S);
        \draw (N) -- (O) -- (Q) -- (R) -- (P) -- (O);
        \draw[dashed] (X) -- (1) -- (W);
        \draw[dashed] (T) -- (2) -- (3) -- (S);
        \draw[dashed] (4) -- (5) -- (6) -- (7);
        \draw[dashed] (8) -- (R);
        \draw[dashed] (T) -- (4);
        \draw[dashed] (1) -- (2) -- (5) -- (8);
        \draw[dashed] (W) -- (3) -- (6) -- (R);
        \draw[dashed] (S) -- (7) -- (Q);
        
        \node (A) at (0,-11.5){{\tiny $a$}};
        \node (B) at (1,-12){{\tiny $b$}};
        \node (C) at (0,-12.5){{\tiny $c$}};
        \node (D) at (2,-12){{\tiny $d$}};
        \node (F) at (4,-12){{\tiny $f$}};
        \node (G) at (5,-12){{\tiny $g$}};
        \node (9) at (3,-12){};
        
        
        \draw (B) -- (D) (G) -- (F) (A) -- (B) -- (C);
        \draw[dashed] (D) -- (9) -- (F);
        
        \draw [
    decoration={
        brace, mirror
    },
    decorate
] (2,-12.3) -- (5,-12.3);
        \node at (3.5,-12.85){{\footnotesize $k$ nodes (odd)}};
        
        
        
        
    \end{tikzpicture}};
    
    \node at (11,0.75){\begin{tikzpicture}[scale=0.475]        
        \node[draw] (AA1) at (-11,17){{\tiny $a$}};
        \node[draw] (BB1) at (-10,16){{\tiny $b$}};
        \node[draw] (CC1) at (-9,15){{\tiny $c$}};
        \node[draw] (DD1) at (-8,14){{\tiny $d$}};
        \node[draw] (EE1) at (-7,13){{\tiny $e$}};
        \node[draw] (FF1) at (-10,14){{\tiny $f$}};
        \node[draw] (CC2) at (-9,13){{\tiny $c$}};
        \node[draw] (DD2) at (-8,12){{\tiny $d$}};
        \node[draw] (BB2) at (-10,12){{\tiny $b$}};
        \node[draw] (AA2) at (-11,11){{\tiny $a$}};
        \node[draw] (CC3) at (-9,11){{\tiny $c$}};
        \node[draw] (BB3) at (-10,10){{\tiny $b$}};
        \node[draw] (FF2) at (-8,10){{\tiny $f$}};
        \node[draw] (CC4) at (-9,9){{\tiny $c$}};
        \node[draw] (DD3) at (-8,8){{\tiny $d$}};
        \node[draw] (EE2) at (-7,7){{\tiny $e$}};
        
        \draw (AA1) -- (BB1) -- (CC1) -- (DD1) -- (EE1) (FF1) -- (CC2) -- (DD2) (BB2) -- (CC3) -- (FF2) (AA2) -- (BB3) -- (CC4) -- (DD3) -- (EE2);
        
        \draw (CC1) -- (FF1) (DD1) -- (CC2) -- (BB2) -- (AA2) (EE1) -- (DD2) -- (CC3) -- (BB3) (FF2) -- (CC4) (FF2) -- (CC4);
        
        \node (AA) at (-11,5.5){{\tiny $a$}};
        \node (BB) at (-10,5.5){{\tiny $b$}};
        \node (CC) at (-9,5.5){{\tiny $c$}};
        \node (DD) at (-8,5.5){{\tiny $d$}};
        \node (EE) at (-7,5.5){{\tiny $e$}};
        \node (FF) at (-9,6.5){{\tiny $f$}};
        
        \draw (AA) -- (BB) -- (CC) -- (DD) -- (EE) (CC) -- (FF); 
        
    \end{tikzpicture}};
    
    \node at (11,-7.25){\begin{tikzpicture}[scale=0.475]
        \node[draw] (A1) at (0,20){{\tiny $a$}};
        \node[draw] (B1) at (1,19){{\tiny $b$}};
        \node[draw] (C1) at (2,18){{\tiny $c$}};
        \node[draw] (D1) at (3,17){{\tiny $d$}};
        \node[draw] (E1) at (4,16){{\tiny $e$}};
        \node[draw] (F1) at (5,15){{\tiny $f$}};
        \node[draw] (G1) at (2,16){{\tiny $g$}};
        \node[draw] (D2) at (3,15){{\tiny $d$}};
        \node[draw] (E2) at (4,14){{\tiny $e$}};
        \node[draw] (C2) at (2,14){{\tiny $c$}};
        \node[draw] (B2) at (1,13){{\tiny $b$}};
        \node[draw] (A2) at (0,12){{\tiny $a$}};
        \node[draw] (D3) at (3,13){{\tiny $d$}};
        \node[draw] (C3) at (2,12){{\tiny $c$}};
        \node[draw] (B3) at (1,11){{\tiny $b$}};
        \node[draw] (G2) at (4,12){{\tiny $g$}};
        \node[draw] (D4) at (3,11){{\tiny $d$}};
        \node[draw] (C4) at (2,10){{\tiny $c$}};
        \node[draw] (E3) at (4,10){{\tiny $e$}};
        \node[draw] (F2) at (5,9){{\tiny $f$}};
        \node[draw] (D5) at (3,9){{\tiny $d$}};
        \node[draw] (E4) at (4,8){{\tiny $e$}};
        \node[draw] (G3) at (2,8){{\tiny $g$}};
        \node[draw] (D6) at (3,7){{\tiny $d$}};
        \node[draw] (C5) at (2,6){{\tiny $c$}};
        \node[draw] (B4) at (1,5){{\tiny $b$}};
        \node[draw] (A3) at (0,4){{\tiny $a$}};
        
        \draw (A1) -- (B1) -- (C1) -- (D1) -- (E1) -- (F1) (G1) -- (D2) -- (E2) (C2) -- (D3) -- (G2) (B2) -- (C3) -- (D4) -- (E3) -- (F2) (A2) -- (B3) -- (C4) -- (D5) -- (E4) (G3) -- (D6);
        \draw (D1) -- (G1) (E1) -- (D2) -- (C2) -- (B2) -- (A2) (F1) -- (E2) -- (D3) -- (C3) -- (B3) (G2) -- (D4) -- (C4) (E3) -- (D5) -- (G3) (F2) -- (E4) -- (D6) -- (C5) -- (B4) -- (A3);
        
        \node (A) at (0,2.5){{\tiny $a$}};
        \node (B) at (1,2.5){{\tiny $b$}};
        \node (C) at (2,2.5){{\tiny $c$}};
        \node (D) at (3,2.5){{\tiny $d$}};
        \node (E) at (4,2.5){{\tiny $e$}};
        \node (F) at (5,2.5){{\tiny $f$}};
        \node (G) at (3,3.5){{\tiny $g$}};
        
        \draw (A) -- (B) -- (C) -- (D) -- (E) -- (F) (D) -- (G);
    \end{tikzpicture}};
    
    \node at (0,-3.75){(a)};
    \node at (6,-3.5){(b)};
    \node at (3,-11.5){(c)};
    \node at (11,-2.5){(d)};
    \node at (11,-11.9){(e)};
        \end{tikzpicture}
    \caption{The $\Gamma$-colored minuscule posets of Proposition \ref{PropNewSL} with top trees $T = Y(i;j,k)$ for $i \ge 1$ and $k \ge j \ge 1$.  (a) Vertically aligned elements have the same color.  The color $h$ is the $k + 1^{\text{st}}$ node from the left in $\Gamma$.  If $j = k$, then $e = h$.  (b) Vertically aligned elements have the same color, except for the unique element of color $f$.  (c) There are two possibilities depending on whether $k$ is even or odd.  Vertically aligned elements have the same color except for the leftmost elements, which alternate between colors $a$ and $c$.}
    \label{FigMinSL}
\end{figure}
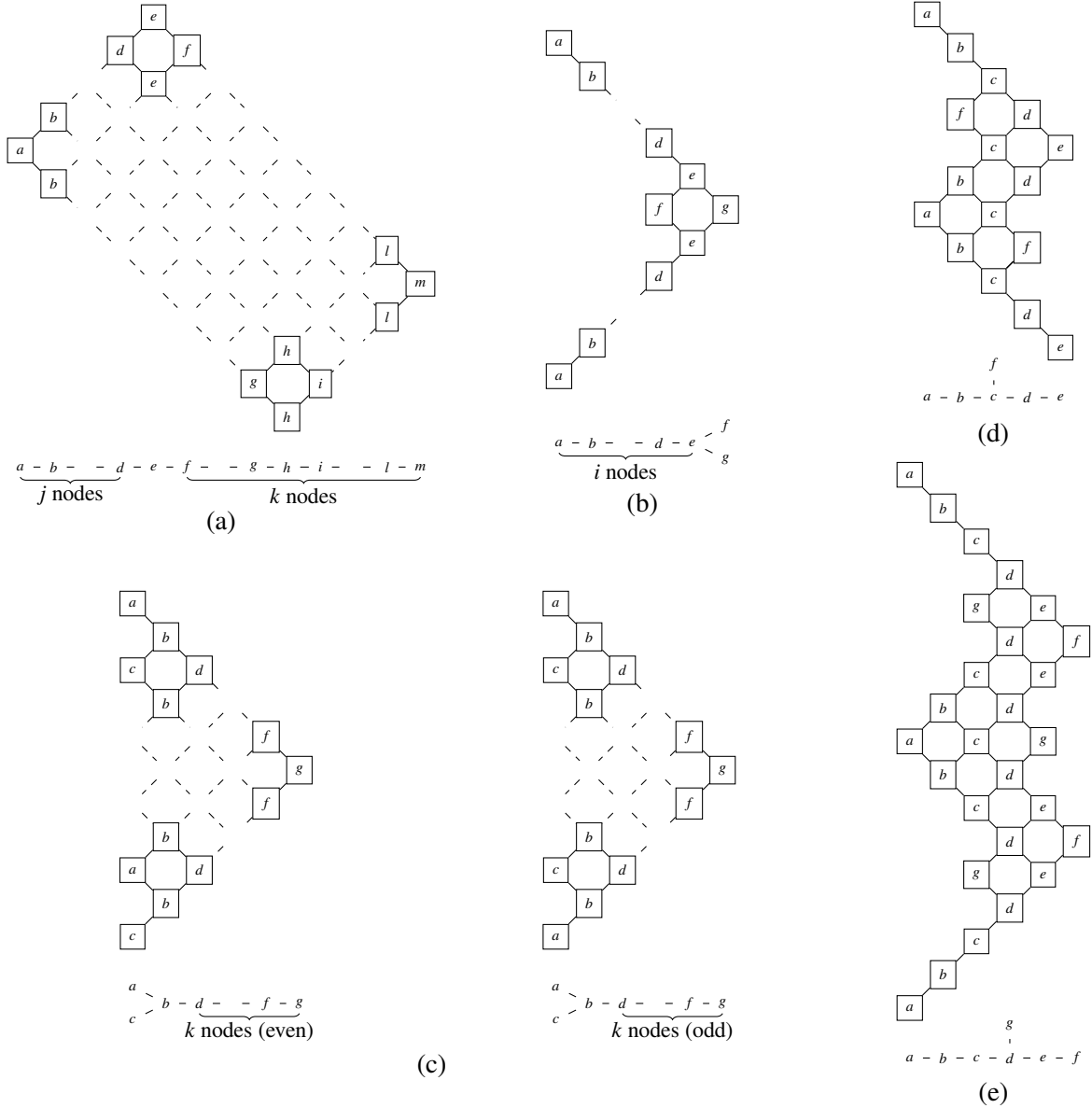

Having obtained connected finite $\Gamma$-colored minuscule posets in the previous result, we now examine three cases in which no $\Gamma$-colored minuscule posets exist.



\begin{proposition}\label{PropNewSLFail}
Fix integers $i \ge 1$ and $k \ge j \ge 1$ and let $T$ be the unique $\Gamma$-colored $d$-complete poset of shape $Y(i;j,k)$ colored by simply laced $\Gamma$.  Then there are no $\Gamma$-colored minuscule posets with top tree $T$ in each of the following cases:
\begin{enumerate}[(a),nosep]
    \item When $i > 1$ and $k \ge j > 1$,
    \item When $i > 2$ and $j = 1$ and $k > 2$, and
    \item When $i > 4$ and $j = 1$ and $k = 2$.
\end{enumerate}
\end{proposition}

\begin{proof}
Following the extension process described in Section \ref{SectionExtending}, we produced a necessary portion of the unique maximal rank complete $\Gamma$-colored $d$-complete poset $P$ with the specified top tree $T$ for each case.  The results for (a), (b), and (c) are respectively displayed in Figure \ref{FigMinSLFail}(a), (b), and (c).  Color extension sets obtained in the process are also displayed.  These extension processes terminated respectively at the third, fifth, and ninth assessment steps.  For each $P$ produced, there is a color $q$ for which $L_q(P) = 3$.  Hence Theorem \ref{ThmAlgorithmExtension}(b) shows there is no $\Gamma$-colored minuscule poset with top tree $T$ in these three cases.
\end{proof}


\begin{figure}
    \centering
    \begin{tikzpicture}
        \node at (0,0){\begin{tikzpicture}[scale=0.475]
        \node (1) at (2,1.5){};
        \node (C) at (3,1.5){{\tiny $a$}};
        \node (D) at (4,1){{\tiny $b$}};
        \node (E) at (3,0.5){{\tiny $e$}};
        \node (F) at (2,0.5){{\tiny $f$}};
        \node (2) at (1,0.5){};
        \node (L) at (5,1){{\tiny $c$}};
        \node (M) at (6,1){{\tiny $d$}};
        \node (3) at (7,1){};
        
        \draw[dotted,thick] (1) -- (C);
        \draw  (C) -- (D);
        \draw (D) -- (E) -- (F) (D) -- (L) -- (M);
        \draw[dotted,thick] (M) -- (3) (2) -- (F);
        
        \node (21) at (1,3){};
        \node[draw] (F1) at (2,4){{\tiny $f$}};
        \node[draw] (E1) at (3,5){{\tiny $e$}};
        \node[draw] (D1) at (4,6){{\tiny $b$}};
        \node [draw] (C1) at (3,7){{\tiny $a$}};
        \node (11) at (2,8){};
        \node [draw] (L1) at (5,5){{\tiny $c$}};
        \node [draw] (M1) at (6,4){{\tiny $d$}};
        \node (31) at (7,3){};
        \node [draw,gray] (D2) at (4,4){{\tiny $b$}};
        \node [draw, gray] (E2) at (3,3){{\tiny $e$}};
        \node [draw, gray] (L2) at (5,3){{\tiny $c$}};
        \node [draw, gray] (C2) at (4,3){{\tiny $a$}};
        
        \draw[dashed] (11) -- (C1) (F1) -- (21) (M1) -- (31);
        \draw (C1) -- (D1) -- (L1) -- (M1) (D1) -- (E1) -- (F1);
        \draw[gray] (E1) -- (D2) -- (L1) (D2) -- (E2) (D2) -- (C2) (D2) -- (L2) (F1) -- (E2) (M1) -- (L2);
        
        \node at (9,4){{\tiny $E_1(T) = \{b\}$}};
        \node at (9,3){{\tiny $E_2(T) = \{e,a,c\}$}};
    \end{tikzpicture}};
    
    \node at (5,0){\begin{tikzpicture}[scale=0.475]
        \node (1) at (0,-0.5){};
        \node (C) at (1,-0.5){{\tiny $a$}};
        \node (D) at (2,-0.5){{\tiny $b$}};
        \node (E) at (3,-0.5){{\tiny $c$}};
        \node (F) at (4,-0.5){{\tiny $d$}};
        \node (G) at (5,-0.5){{\tiny $e$}};
        \node (H) at (6,-0.5){{\tiny $f$}};
        \node (2) at (7,-0.5){};
        \node (M) at (3,0.5){{\tiny $g$}};
        
        \draw[dotted,thick] (1) -- (C) (H) -- (2);
        \draw (C) -- (D) -- (E) -- (F) -- (G) -- (H) (E) -- (M);
        
        \node (11) at (0,10){};
        \node[draw] (C1) at (1,9){{\tiny $a$}};
        \node[draw] (D1) at (2,8){{\tiny $b$}};
        \node[draw] (E1) at (3,7){{\tiny $c$}};
        \node[draw] (F1) at (4,6){{\tiny $d$}};
        \node[draw] (G1) at (5,5){{\tiny $e$}};
        \node[draw] (H1) at (6,4){{\tiny $f$}};
        \node (21) at (7,3){};
        \node[draw] (M1) at (2,6){{\tiny $g$}};
        \node[draw,opacity=.5] (E2) at (3,5){{\tiny $c$}};
        \node[draw,opacity=.5] (F2) at (4,4){{\tiny $d$}};
        \node[draw,opacity=.5] (G2) at (5,3){{\tiny $e$}};
        \node[draw,opacity=.5] (D2) at (2,4){{\tiny $b$}};
        \node (12) at (0,2){};
        \node[draw,opacity=.5] (C2) at (1,3){{\tiny $a$}};
        \node[draw,opacity=.5] (E3) at (3,3){{\tiny $c$}};
        \node[draw,opacity=.5] (D3) at (2,2){{\tiny $b$}};
        \node[draw,opacity=.5] (F3) at (4,2){{\tiny $d$}};
        \node[draw,opacity=.5] (M2) at (3,2){{\tiny $g$}};
        \node (H2) at (6,2){};
        
        \draw[dashed] (11) -- (C1);
        \draw[dashed, opacity=.5] (12) -- (C2);
        \draw[dashed] (H1) -- (21);
        \draw[dashed,opacity=.5] (G2) -- (H2) -- (21);
        \draw (C1) -- (D1) -- (E1) -- (M1) (E1) -- (F1) -- (G1) -- (H1);
        \draw[opacity=.5] (M1) -- (E2) -- (F1) (E2) -- (F2) -- (G1) (F2) -- (G2) -- (H1) (E2) -- (D2) -- (E3) -- (F2) (D2) -- (C2) -- (D3) -- (E3) -- (M2) (E3) -- (F3) -- (G2);
        
        \node at (9,5){{\tiny $E_1(T) = \{c\}$}};
        \node at (9,4){{\tiny $E_2(T) = \{b,d\}$}};
        \node at (9,3){{\tiny $E_3(T) \supseteq \{a,c,e\}$}};
        \node at (9,2){{\tiny $E_4(T) \supseteq \{b,g,d\}$}};

    \end{tikzpicture}};
    
    \node at (10.5,0){\begin{tikzpicture}[scale=.475]
        \node (11) at (1,19){};
        \node[draw] (C1) at (2,18){{\tiny $a$}};
        \node[draw] (D1) at (3,17){{\tiny $b$}};
        \node[draw] (E1) at (4,16){{\tiny $c$}};
        \node[draw] (F1) at (5,15){{\tiny $d$}};
        \node[draw] (G1) at (6,14){{\tiny $e$}};
        \node[draw] (H1) at (7,13){{\tiny $f$}};
        \node[draw] (L1) at (8,12){{\tiny $g$}};
        \node[draw] (M1) at (5,13){{\tiny $h$}};
        
        \draw [dashed] (11) -- (C1);
        \draw (C1) -- (D1) -- (E1) -- (F1) -- (G1) -- (H1) -- (L1) (G1) -- (M1);
        
        \node[draw,gray] (G2) at (6,12){{\tiny $e$}};
        \node[draw,gray] (H2) at (7,11){{\tiny $f$}};
        \node[draw,gray] (F2) at (5,11){{\tiny $d$}};
        \node[draw,gray] (E2) at (4,10){{\tiny $c$}};
        \node[draw,gray] (D2) at (3,9){{\tiny $b$}};
        \node[draw,gray] (C2) at (2,8){{\tiny $a$}};
        \node[gray] (12) at (1,7){};
        \node[draw,gray] (G3) at (6,10){{\tiny $e$}};
        \node[draw,gray] (M2) at (7,9){{\tiny $h$}};
        \node[draw,gray] (F3) at (5,9){{\tiny $d$}};
        \node[draw,gray] (E3) at (4,8){{\tiny $c$}};
        \node[draw,gray] (G4) at (6,8){{\tiny $e$}};
        \node[draw,gray] (D3) at (3,7){{\tiny $b$}};
        \node[draw,gray] (F4) at (5,7){{\tiny $d$}};
        \node[draw,gray] (H3) at (7,7){{\tiny $f$}};
        \node[draw,gray] (L2) at (8,6){{\tiny $g$}};
        \node[draw,gray] (E4) at (4,6){{\tiny $c$}};
        \node[draw,gray] (F5) at (5,5){{\tiny $d$}};
        \node[draw,gray] (G5) at (6,6){{\tiny $e$}};
        \node[draw,gray] (H4) at (7,5){{\tiny $f$}};
        \node[draw,gray] (M3) at (6,5){{\tiny $h$}};
        \node[gray] (C3) at (2,6){};
        \node[gray] (D4) at (3,5){};
        \node[gray] (13) at (1,5){};
        
        \draw[gray] (M1) -- (G2) -- (H2) (F2) -- (G3) -- (M2) (E2) -- (F3) -- (G4) -- (H3) -- (L2) (D2) -- (E3) -- (F4) -- (G5) -- (H4) (C2) -- (D3) -- (E4) -- (F5);
        \draw[gray] (H1) -- (G2) -- (F2) -- (E2) -- (D2) -- (C2) (L1) -- (H2) -- (G3) -- (F3) -- (E3) -- (D3) (M2) -- (G4) -- (F4) -- (E4) (H3) -- (G5) -- (F5) (L2) -- (H4) (G5) -- (M3);
        \draw[gray,dashed] (C2) -- (12) (D3) -- (C3) (E4) -- (D4);
        \draw[gray,dashed] (12) -- (C3) -- (D4);
        
        \node (1) at (1,2.5){};
        \node (C) at (2,2.5){{\tiny $a$}};
        \node (D) at (3,2.5){{\tiny $b$}};
        \node (E) at (4,2.5){{\tiny $c$}};
        \node (F) at (5,2.5){{\tiny $d$}};
        \node (G) at (6,2.5){{\tiny $e$}};
        \node (H) at (7,2.5){{\tiny $f$}};
        \node (L) at (8,2.5){{\tiny $g$}};
        \node (M) at (6,3.5){{\tiny $h$}};
        
        \draw[dotted,thick] (1) -- (C);
        \draw (C) -- (D) -- (E) -- (F) -- (G) -- (H) -- (L) (G) -- (M);
        
        \node at (11,12){{\tiny $E_1(T) = \{e\}$}};
        \node at (11,11){{\tiny $E_2(T) = \{d,f\}$}};
        \node at (11,10){{\tiny $E_3(T) = \{c,e\}$}};
        \node at (11,9){{\tiny $E_4(T) = \{b,d,h\}$}};
        \node at (11,8){{\tiny $E_5(T) = \{a,c,e\}$}};
        \node at (11,7){{\tiny $E_6(T) \supseteq \{b,d,f\}$}};
        \node at (11,6){{\tiny $E_7(T) \supseteq \{c,e,g\}$}};
        \node at (11,5){{\tiny $E_8(T) \supseteq \{d,h,f\}$}};
    \end{tikzpicture}};
    
    \node at (0,-2){(a)};
    \node at (5,-3){(b)};
    \node at (11,-4.5){(c)};
    \end{tikzpicture} 
    \caption{Necessary portions of the unique maximal rank complete $\Gamma$-colored $d$-complete posets $P$ of Proposition \ref{PropNewSLFail} with top trees $T = Y(i;j,k)$ for $i \ge 1$ and $k \ge j \ge 1$.  Gray elements and edges are not part of the top trees.  Color extension sets (or subsets thereof) are written in line with extending elements of the corresponding colors and ranks.  (a) We have $L_b(P) = 3$.  (b) We have $L_c(P) = 3$.  (c) We have $L_e(P) = 3$.}
    \label{FigMinSLFail}
\end{figure}
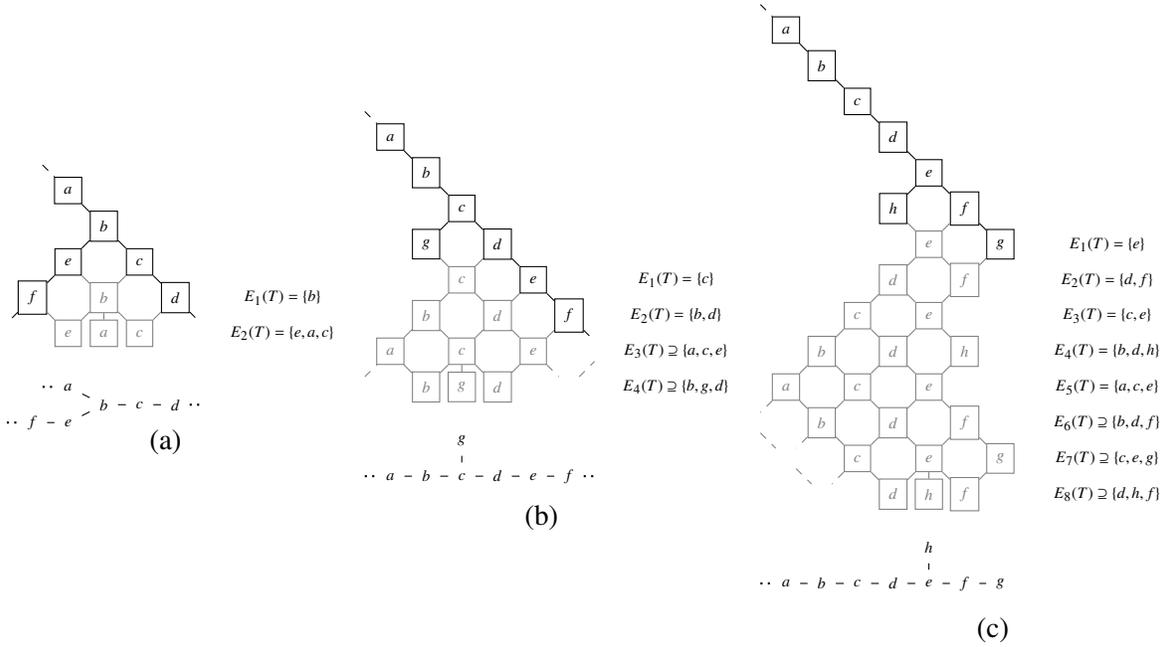

We can now classify the slant irreducible $\Gamma$-colored minuscule posets when $\Gamma$ is simply laced and contains more than one color.  To state this classification concisely, we first name the types of $\Gamma$-colored posets appearing in it.


\begin{definition}
Let $P$ be a finite $\Gamma$-colored poset.
\begin{enumerate}[(a),nosep]
    \item If $P$ is isomorphic to a poset of the form displayed in Figure \ref{FigMinSL}(a), then $P$ has \emph{type $A$ exterior}.
    \item If $P$ is isomorphic to a poset of the form displayed in Figure \ref{FigMinSL}(b), then $P$ has \emph{type $D$ standard}.
    \item If $P$ is isomorphic to a poset of the form displayed in Figure \ref{FigMinSL}(c), then $P$ has \emph{type $D$ spin}.
    \item If $P$ is isomorphic to the poset displayed in Figure \ref{FigMinSL}(d), then $P$ has \emph{type $E_6$}.
    \item If $P$ is isomorphic to the poset displayed in Figure \ref{FigMinSL}(e), then $P$ has \emph{type $E_7$}.
\end{enumerate}
\end{definition}


\noindent Now we obtain the main result of this section.


\begin{theorem}\label{ThmSimpLacedSlantIrred}
Let $P$ be a connected finite $\Gamma$-colored minuscule poset and assume $\Gamma$ is simply laced and contains more than one color.  If $P$ is slant irreducible as a $\Gamma$-colored $d$-complete poset, then $P$ has type A exterior, type $D$ standard, type $D$ spin, type $E_6$, or type $E_7$.
\end{theorem}

\begin{proof}
Since $\Gamma$ contains more than one color, so does the top tree $T$.  Hence Corollary \ref{CorTHasShapeY} shows $T$ has shape $Y(i;j,k)$ for some integers $i \ge 1$ and $k \ge j \ge 1$.



If $i = 1$, then Proposition \ref{PropNewSL}(a) shows $P$ has type $A$ exterior.  Now assume $i > 1$.  Proposition \ref{PropNewSLFail}(a) shows that $j = 1$.  If $k = 1$, then Proposition \ref{PropNewSL}(b) shows $P$ has type $D$ standard.  Now assume $k > 1$.  If $i = 2$, then Proposition \ref{PropNewSL}(c) shows $P$ has type $D$ spin.  Now assume $i > 2$.  Proposition \ref{PropNewSLFail}(b) shows $k = 2$ and Proposition \ref{PropNewSLFail}(c) shows $i = 3$ or $i = 4$.  Proposition \ref{PropNewSL}(d) (respectively \ref{PropNewSL}(e)) shows that $P$ has type $E_6$ (respectively $E_7$) in this case.
\end{proof}


\section{Main classification results}\label{SectionMainResults}


Our main classification result is presented in Theorem \ref{ThmMinusculeClassify}, where we classify all $\Gamma$-colored minuscule posets.  This result is aided by Proposition 25 of \cite{dC-class}, which shows how a $\Gamma$-colored minuscule poset decomposes into a disjoint union of connected $\Gamma$-colored minuscule posets.  We then give an overview of representation definitions given in \cite{Unify} and apply \cite[Thm. 38(b)]{Unify} to obtain the classification of $P$-minuscule representations in Theorem \ref{ThmPMinusculeClassify}.




The connected finite $\Gamma$-colored minuscule posets appearing in Theorem \ref{ThmMinusculeClassify}(i) are Proctor's colored minuscule posets.  
After defining uncolored minuscule posets using weight diagrams of minuscule representations, he then identified them with certain posets of coroots in the associated finite Lie types and used this identification to produce their original coloring maps.  We recover these realizations in Section \ref{SectionHistory} in our setting using the classification of Theorem \ref{ThmMinusculeClassify}(i).
See Figure \ref{FigFiniteMinusculeHasseDiagrams} for the Hasse diagrams of all connected finite $\Gamma$-colored minuscule posets and their corresponding Dynkin diagrams of finite Lie type.

\begin{figure}[t!]
\centering
    \begin{tikzpicture}[scale=.965]
        \node at (-1,1.35){{\footnotesize Type $A$ standard}};
        \node at (-1,1){{\footnotesize $|\Gamma| > 0$}};
        
        \node at (5.75,2.35){{\footnotesize Type $A$ exterior}};
        \node at (5.75,2){{\footnotesize $|\Gamma| > 2$}};
        
        \node at (-1.25,-5){{\footnotesize Type $B$}};
        \node at (-1.25,-5.35){{\footnotesize $|\Gamma| > 1$}};
        
        \node at (2.75,-5){{\footnotesize Type $C$}};
        \node at (2.75,-5.35){{\footnotesize $|\Gamma| > 2$}};
        
        \node at (7.25,-5){{\footnotesize Type $D$ standard}};
        \node at (7.25,-5.35){{\footnotesize $|\Gamma| > 3$}};
        
        \node at (1.1,-12.3){{\footnotesize Type $D$ spin}};
        \node at (1.1,-12.65){{\footnotesize $|\Gamma| > 4$, odd}};
        
        \node at (6.1,-12.3){{\footnotesize Type $D$ spin}};
        \node at (6.1,-12.65){{\footnotesize $|\Gamma| > 4$, even}};
        
        \node at (11.5,.5){{\footnotesize Type $E_6$}};
        
        \node at (11.5,-8){{\footnotesize Type $E_7$}};
    
        \node at (-2,0){\begin{tikzpicture}[scale=.475]
        \node (A) at (0,0){{\tiny $a$}};
        \node (B) at (1,0){{\tiny $b$}};
        \node (1) at (2,0){};
        \node (D) at (3,0){{\tiny $d$}};
        \node (E) at (4,0){{\tiny $e$}};
        
        \draw (A) -- (B) (D) -- (E);
        \draw[dashed] (B) -- (1) -- (D);
        
        \node[draw] (Z) at (0,5.5){{\tiny $a$}};
        \node[draw] (Y) at (1,4.5){{\tiny $b$}};
        \node (2) at (2,3.5){};
        \node[draw] (X) at (3,2.5){{\tiny $d$}};
        \node[draw] (W) at (4,1.5){{\tiny $e$}};
        
        \draw (Z) -- (Y) (X) -- (W);
        \draw[dashed] (Y) -- (2) -- (X);
    \end{tikzpicture}};
    
    \node at (4,0){\begin{tikzpicture}[scale=.475]
        \node (A) at (0,-3){{\tiny $a$}};
        \node (B) at (1,-3){{\tiny $b$}};
        \node (1) at (2,-3){};
        \node (D) at (3,-3){{\tiny $d$}};
        \node (E) at (4,-3){{\tiny $e$}};
        \node (F) at (5,-3){{\tiny $f$}};
        \node (2) at (6,-3){};
        \node (G) at (7,-3){{\tiny $g$}};
        \node (H) at (8,-3){{\tiny $h$}};
        \node (I) at (9,-3){{\tiny $i$}};
        \node (3) at (10,-3){};
        \node (L) at (11,-3){{\tiny $l$}};
        \node (M) at (12,-3){{\tiny $m$}};
        
        
        \draw (A) -- (B) (D) -- (E) -- (F) (G) -- (H) -- (I) (L) -- (M);
        \draw[dashed] (B) -- (1) -- (D) (F) -- (2) -- (G) (I) -- (3) -- (L);
        
        \node[draw] (E1) at (4,10.5){{\tiny $e$}};
        \node[draw] (F1) at (5,9.5){{\tiny $f$}};
        \node (21) at (6,8.5){};
        \node (G1) at (7,7.5){};
        \node (H1) at (8,6.5){};
        \node (I1) at (9,5.5){};
        \node (31) at (10,4.5){};
        \node[draw] (L1) at (11,3.5){{\tiny $l$}};
        \node[draw] (M1) at (12,2.5){{\tiny $m$}};
        
        \node[draw] (D1) at (3,9.5){{\tiny $d$}};
        \node[draw] (E2) at (4,8.5){{\tiny $e$}};
        \node (F2) at (5,7.5){};
        \node (22) at (6,6.5){};
        \node (G2) at (7,5.5){};
        \node (H2) at (8,4.5){};
        \node (I2) at (9,3.5){};
        \node (32) at (10,2.5){};
        \node[draw] (L2) at (11,1.5){{\tiny $l$}};
        
        \node (11) at (2,8.5){};
        \node (D2) at (3,7.5){};
        \node (E3) at (4,6.5){};
        \node (F3) at (5,5.5){};
        \node (23) at (6,4.5){};
        \node (G3) at (7,3.5){};
        \node (H3) at (8,2.5){};
        \node (I3) at (9,1.5){};
        \node (33) at (10,0.5){};
        
        \node[draw] (B1) at (1,7.5){{\tiny $b$}};
        \node (12) at (2,6.5){};
        \node (D3) at (3,5.5){};
        \node (E4) at (4,4.5){};
        \node (F4) at (5,3.5){};
        \node (24) at (6,2.5){};
        \node (G4) at (7,1.5){};
        \node[draw] (H4) at (8,0.5){{\tiny $h$}};
        \node[draw] (I4) at (9,-0.5){{\tiny $i$}};
        
        \node[draw] (A1) at (0,6.5){{\tiny $a$}};
        \node[draw] (B2) at (1,5.5){{\tiny $b$}};
        \node (13) at (2,4.5){};
        \node (D4) at (3,3.5){};
        \node (E5) at (4,2.5){};
        \node (F5) at (5,1.5){};
        \node (25) at (6,0.5){};
        \node[draw] (G5) at (7,-0.5){{\tiny $g$}};
        \node[draw] (H5) at (8,-1.5){{\tiny $h$}};

        \node (9) at (2,8.5){};
        \node (9) at (2,8.5){};
        

    \draw (E1) -- (F1) -- (E2) -- (D1) -- (E1) (B1) -- (A1) -- (B2) (G5) -- (H5) -- (I4) -- (H4) -- (G5) (L1) -- (M1) -- (L2);
    
    \draw[dashed] (F1) -- (21) -- (G1) -- (H1) -- (I1) -- (31) -- (L1) (E2) -- (F2) -- (22) -- (G2) -- (H2) -- (I2) -- (32) -- (L2) (11) -- (D2) -- (E3) -- (F3) -- (23) -- (G3) -- (H3) -- (I3) -- (33) (B1) -- (12) -- (D3) -- (E4) -- (F4) -- (24) -- (G4) -- (H4) (B2) -- (13) -- (D4) -- (E5) -- (F5) -- (25) -- (G5);
    \draw[dashed] (D1) -- (11) -- (B1) (E2) -- (D2) -- (12) -- (B2) (21) -- (F2) -- (E3) -- (D3) -- (13) (G1) -- (22) -- (F3) -- (E4) -- (D4) (H1) -- (G2) -- (23) -- (F4) -- (E5) (I1) -- (H2) -- (G3) -- (24) -- (F5) (31) -- (I2) -- (H3) -- (G4) -- (25) (L1) -- (32) -- (I3) -- (H4) (L2) -- (33) -- (I4);
    \end{tikzpicture}};
    
    \node at (-2,-7.5){\begin{tikzpicture}[scale=.475]
    \tikzset{edge/.style = {->,> = latex'}}
        \node[draw] (Z) at (0,0){{\tiny $a$}};
        \node[draw] (Y) at (1,-1){{\tiny $b$}};
        \node[draw] (X) at (2,-2){{\tiny $c$}};
        \node[draw] (W) at (4,-4){{\tiny $e$}};
        \node[draw] (V) at (5,-5){{\tiny $f$}};
        
        \node[draw] (U) at (0,-2){{\tiny $a$}};
        \node[draw] (T) at (1,-3){{\tiny $b$}};
        \node[draw] (S) at (4,-6){{\tiny $e$}};
        
        \node[draw] (R) at (1,-7){{\tiny $b$}};
        \node[draw] (Q) at (2,-8){{\tiny $c$}};
        
        \node[draw] (P) at (0,-8){{\tiny $a$}};
        \node[draw] (O) at (1,-9){{\tiny $b$}};
        \node[draw] (N) at (0,-10){{\tiny $a$}};
        
        \node (1) at (3,-3){};
        \node (2) at (2,-4){};
        \node (3) at (3,-5){};
        \node (4) at (0,-4){};
        \node (5) at (1,-5){};
        \node (6) at (2,-6){};
        \node (7) at (3,-7){};
        \node (8) at (0,-6){};
        
        \draw (Z) -- (Y) -- (X) -- (T) -- (U) -- (Y);
        \draw (W) -- (V) -- (S);
        \draw (N) -- (O) -- (Q) -- (R) -- (P) -- (O);
        \draw[dashed] (X) -- (1) -- (W);
        \draw[dashed] (T) -- (2) -- (3) -- (S);
        \draw[dashed] (4) -- (5) -- (6) -- (7);
        \draw[dashed] (8) -- (R);
        \draw[dashed] (T) -- (4);
        \draw[dashed] (1) -- (2) -- (5) -- (8);
        \draw[dashed] (W) -- (3) -- (6) -- (R);
        \draw[dashed] (S) -- (7) -- (Q);
        
        \node (A) at (0,-11.5){{\tiny $a$}};
        \node (B) at (1,-11.5){{\tiny $b$}};
        \node (C) at (2,-11.5){{\tiny $c$}};
        \node (E) at (4,-11.5){{\tiny $e$}};
        \node (F) at (5,-11.5){{\tiny $f$}};
        \node (9) at (3,-11.5){};
        
        \node (AT) at (0,-11.4){};
        \node (BT) at (1,-11.4){};
        \node (AB) at (0,-11.6){};
        \node (BB) at (1,-11.6){};
        
        \draw (B) -- (C) (E) -- (F);
        \draw[dashed] (C) -- (9) -- (E);
        \draw[edge] (BT) to (AT);
        \draw[edge] (AB) to (BB);
        
        \node at (.5,-11.1){{\tiny $2$}};
        \node at (.5,-11.9){{\tiny $1$}};
        
    \end{tikzpicture}};
    
    \node at (2.5,-7.5){\begin{tikzpicture}[scale=.475]
    \tikzset{edge/.style = {->,> = latex'}}
        \node[draw] (Z) at (0,0){{\tiny $a$}};
        \node[draw] (Y) at (1,-1){{\tiny $b$}}; 
        \node (3) at (2,-2){};
        \node[draw] (X) at (3,-3){{\tiny $d$}};
        \node[draw] (W) at (4,-4){{\tiny $e$}};
        \node[draw] (V) at (5,-5){{\tiny $f$}};
        \node (4) at (7,-7){};
        \node[draw] (R) at (4,-6){{\tiny $e$}};
        \node (P) at (6,-8){};
        \node (5) at (7,-9){};
        \node[draw] (N) at (3,-7){{\tiny $d$}};
        \node (6) at (2,-8){};
        \node[draw] (M) at (1,-9){{\tiny $b$}};
        \node[draw] (L) at (0,-10){{\tiny $a$}};
        
        \node (A) at (0,-11.5){{\tiny $a$}};
        \node (B) at (1,-11.5){{\tiny $b$}};
        \node (1) at (2,-11.5){};
        \node (D) at (3,-11.5){{\tiny $d$}};
        \node (E) at (4,-11.5){{\tiny $e$}};
        \node (F) at (5,-11.5){{\tiny $f$}};
        
        \node (ET) at (4,-11.4){};
        \node (EB) at (4,-11.6){};
        \node (FT) at (5,-11.4){};
        \node (FB) at (5,-11.6){};
        
        \draw (Z) -- (Y) (X) -- (W) -- (V);
        \draw (V) -- (R);
        \draw (R) -- (N) (M) -- (L);
        \draw[dashed] (Y) -- (3) -- (X);
        \draw [dashed] (N) -- (6) -- (M);
        \draw (A) -- (B) (D) -- (E);
        \draw[dashed] (B) -- (1) -- (D);
        
        \draw [edge] (FT) to (ET);
        \draw [edge] (EB) to (FB);
        
        \node at (4.5,-11.1){{\tiny 2}};
        \node at (4.5,-11.9){{\tiny 1}};
    \end{tikzpicture}};
    
    \node at (6,-7.5){\begin{tikzpicture}[scale=.475]
        \node (A) at (0,0){{\tiny $a$}};
        \node (B) at (1,0){{\tiny $b$}};
        \node (1) at (2,0){};
        \node (D) at (3,0){{\tiny $d$}};
        \node (E) at (4,0){{\tiny $e$}};
        \node (F) at (5,0.5){{\tiny $f$}};
        \node (G) at (5,-0.5){{\tiny $g$}};
        
        \draw (A) -- (B) (D) -- (E) (F) -- (E) -- (G);
        \draw [dashed] (B) -- (1) -- (D);
        
        \node[draw] (A1) at (0,12){{\tiny $a$}};
        \node[draw] (B1) at (1,11){{\tiny $b$}};
        \node (11) at (2,10){};
        \node[draw] (D1) at (3,9){{\tiny $d$}};
        \node[draw] (E1) at (4,8){{\tiny $e$}};
        \node[draw] (F1) at (3,7){{\tiny $f$}};
        \node[draw] (G1) at (5,7){{\tiny $g$}};
        \node[draw] (E2) at (4,6){{\tiny $e$}};
        \node[draw] (D2) at (3,5){{\tiny $d$}};
        \node (12) at (2,4){};
        \node[draw] (B2) at (1,3){{\tiny $b$}};
        \node[draw] (A2) at (0,2){{\tiny $a$}};
        \draw [dashed] (B1) -- (11) -- (D1) (D2) -- (12) -- (B2);
        \draw (A1) -- (B1) (D1) -- (E1) -- (F1) -- (E2) -- (D2) (E1) -- (G1) -- (E2) (B2) -- (A2);
        
    \end{tikzpicture}};
    
    \node at (0,-15){\begin{tikzpicture}[scale=.475]
        \node[draw] (Z) at (0,0){{\tiny $a$}};
        \node[draw] (Y) at (1,-1){{\tiny $b$}};
        \node[draw] (X) at (2,-2){{\tiny $d$}};
        \node[draw] (W) at (4,-4){{\tiny $f$}};
        \node[draw] (V) at (5,-5){{\tiny $g$}};
        
        \node[draw] (U) at (0,-2){{\tiny $c$}};
        \node[draw] (T) at (1,-3){{\tiny $b$}};
        \node[draw] (S) at (4,-6){{\tiny $f$}};
        
        \node[draw] (R) at (1,-7){{\tiny $b$}};
        \node[draw] (Q) at (2,-8){{\tiny $d$}};
        
        \node[draw] (P) at (0,-8){{\tiny $a$}};
        \node[draw] (O) at (1,-9){{\tiny $b$}};
        \node[draw] (N) at (0,-10){{\tiny $c$}};
        
        \node (1) at (3,-3){};
        \node (2) at (2,-4){};
        \node (3) at (3,-5){};
        \node (4) at (0,-4){};
        \node (5) at (1,-5){};
        \node (6) at (2,-6){};
        \node (7) at (3,-7){};
        \node (8) at (0,-6){};
        
        \draw (Z) -- (Y) -- (X) -- (T) -- (U) -- (Y);
        \draw (W) -- (V) -- (S);
        \draw (N) -- (O) -- (Q) -- (R) -- (P) -- (O);
        \draw[dashed] (X) -- (1) -- (W);
        \draw[dashed] (T) -- (2) -- (3) -- (S);
        \draw[dashed] (4) -- (5) -- (6) -- (7);
        \draw[dashed] (8) -- (R);
        \draw[dashed] (T) -- (4);
        \draw[dashed] (1) -- (2) -- (5) -- (8);
        \draw[dashed] (W) -- (3) -- (6) -- (R);
        \draw[dashed] (S) -- (7) -- (Q);
        
        \node (A) at (0,-11.5){{\tiny $a$}};
        \node (B) at (1,-12){{\tiny $b$}};
        \node (C) at (0,-12.5){{\tiny $c$}};
        \node (D) at (2,-12){{\tiny $d$}};
        \node (F) at (4,-12){{\tiny $f$}};
        \node (G) at (5,-12){{\tiny $g$}};
        \node (9) at (3,-12){};
        
        
        \draw (B) -- (D) (G) -- (F) (A) -- (B) -- (C);
        \draw[dashed] (D) -- (9) -- (F);
        
        
        
    \end{tikzpicture}};
    
    \node at (5,-15){\begin{tikzpicture}[scale=.475]
        \node[draw] (Z) at (0,0){{\tiny $a$}};
        \node[draw] (Y) at (1,-1){{\tiny $b$}};
        \node[draw] (X) at (2,-2){{\tiny $d$}};
        \node[draw] (W) at (4,-4){{\tiny $f$}};
        \node[draw] (V) at (5,-5){{\tiny $g$}};
        
        \node[draw] (U) at (0,-2){{\tiny $c$}};
        \node[draw] (T) at (1,-3){{\tiny $b$}};
        \node[draw] (S) at (4,-6){{\tiny $f$}};
        
        \node[draw] (R) at (1,-7){{\tiny $b$}};
        \node[draw] (Q) at (2,-8){{\tiny $d$}};
        
        \node[draw] (P) at (0,-8){{\tiny $c$}};
        \node[draw] (O) at (1,-9){{\tiny $b$}};
        \node[draw] (N) at (0,-10){{\tiny $a$}};
        
        \node (1) at (3,-3){};
        \node (2) at (2,-4){};
        \node (3) at (3,-5){};
        \node (4) at (0,-4){};
        \node (5) at (1,-5){};
        \node (6) at (2,-6){};
        \node (7) at (3,-7){};
        \node (8) at (0,-6){};
        
        \draw (Z) -- (Y) -- (X) -- (T) -- (U) -- (Y);
        \draw (W) -- (V) -- (S);
        \draw (N) -- (O) -- (Q) -- (R) -- (P) -- (O);
        \draw[dashed] (X) -- (1) -- (W);
        \draw[dashed] (T) -- (2) -- (3) -- (S);
        \draw[dashed] (4) -- (5) -- (6) -- (7);
        \draw[dashed] (8) -- (R);
        \draw[dashed] (T) -- (4);
        \draw[dashed] (1) -- (2) -- (5) -- (8);
        \draw[dashed] (W) -- (3) -- (6) -- (R);
        \draw[dashed] (S) -- (7) -- (Q);
        
        \node (A) at (0,-11.5){{\tiny $a$}};
        \node (B) at (1,-12){{\tiny $b$}};
        \node (C) at (0,-12.5){{\tiny $c$}};
        \node (D) at (2,-12){{\tiny $d$}};
        \node (F) at (4,-12){{\tiny $f$}};
        \node (G) at (5,-12){{\tiny $g$}};
        \node (9) at (3,-12){};
        
        
        \draw (B) -- (D) (G) -- (F) (A) -- (B) -- (C);
        \draw[dashed] (D) -- (9) -- (F);
        
        
        
        
    \end{tikzpicture}};
    
    \node at (11,-12){\begin{tikzpicture}[scale=0.475]
        \node[draw] (A1) at (0,20){{\tiny $a$}};
        \node[draw] (B1) at (1,19){{\tiny $b$}};
        \node[draw] (C1) at (2,18){{\tiny $c$}};
        \node[draw] (D1) at (3,17){{\tiny $d$}};
        \node[draw] (E1) at (4,16){{\tiny $e$}};
        \node[draw] (F1) at (5,15){{\tiny $f$}};
        \node[draw] (G1) at (2,16){{\tiny $g$}};
        \node[draw] (D2) at (3,15){{\tiny $d$}};
        \node[draw] (E2) at (4,14){{\tiny $e$}};
        \node[draw] (C2) at (2,14){{\tiny $c$}};
        \node[draw] (B2) at (1,13){{\tiny $b$}};
        \node[draw] (A2) at (0,12){{\tiny $a$}};
        \node[draw] (D3) at (3,13){{\tiny $d$}};
        \node[draw] (C3) at (2,12){{\tiny $c$}};
        \node[draw] (B3) at (1,11){{\tiny $b$}};
        \node[draw] (G2) at (4,12){{\tiny $g$}};
        \node[draw] (D4) at (3,11){{\tiny $d$}};
        \node[draw] (C4) at (2,10){{\tiny $c$}};
        \node[draw] (E3) at (4,10){{\tiny $e$}};
        \node[draw] (F2) at (5,9){{\tiny $f$}};
        \node[draw] (D5) at (3,9){{\tiny $d$}};
        \node[draw] (E4) at (4,8){{\tiny $e$}};
        \node[draw] (G3) at (2,8){{\tiny $g$}};
        \node[draw] (D6) at (3,7){{\tiny $d$}};
        \node[draw] (C5) at (2,6){{\tiny $c$}};
        \node[draw] (B4) at (1,5){{\tiny $b$}};
        \node[draw] (A3) at (0,4){{\tiny $a$}};
        
        \draw (A1) -- (B1) -- (C1) -- (D1) -- (E1) -- (F1) (G1) -- (D2) -- (E2) (C2) -- (D3) -- (G2) (B2) -- (C3) -- (D4) -- (E3) -- (F2) (A2) -- (B3) -- (C4) -- (D5) -- (E4) (G3) -- (D6);
        \draw (D1) -- (G1) (E1) -- (D2) -- (C2) -- (B2) -- (A2) (F1) -- (E2) -- (D3) -- (C3) -- (B3) (G2) -- (D4) -- (C4) (E3) -- (D5) -- (G3) (F2) -- (E4) -- (D6) -- (C5) -- (B4) -- (A3);
        
        \node (A) at (0,2.5){{\tiny $a$}};
        \node (B) at (1,2.5){{\tiny $b$}};
        \node (C) at (2,2.5){{\tiny $c$}};
        \node (D) at (3,2.5){{\tiny $d$}};
        \node (E) at (4,2.5){{\tiny $e$}};
        \node (F) at (5,2.5){{\tiny $f$}};
        \node (G) at (3,3.5){{\tiny $g$}};
        
        \draw (A) -- (B) -- (C) -- (D) -- (E) -- (F) (D) -- (G);
        \end{tikzpicture}};
        
        \node at (11,-2){\begin{tikzpicture}[scale=0.475]
        \node[draw] (AA1) at (-11,17){{\tiny $a$}};
        \node[draw] (BB1) at (-10,16){{\tiny $b$}};
        \node[draw] (CC1) at (-9,15){{\tiny $c$}};
        \node[draw] (DD1) at (-8,14){{\tiny $d$}};
        \node[draw] (EE1) at (-7,13){{\tiny $e$}};
        \node[draw] (FF1) at (-10,14){{\tiny $f$}};
        \node[draw] (CC2) at (-9,13){{\tiny $c$}};
        \node[draw] (DD2) at (-8,12){{\tiny $d$}};
        \node[draw] (BB2) at (-10,12){{\tiny $b$}};
        \node[draw] (AA2) at (-11,11){{\tiny $a$}};
        \node[draw] (CC3) at (-9,11){{\tiny $c$}};
        \node[draw] (BB3) at (-10,10){{\tiny $b$}};
        \node[draw] (FF2) at (-8,10){{\tiny $f$}};
        \node[draw] (CC4) at (-9,9){{\tiny $c$}};
        \node[draw] (DD3) at (-8,8){{\tiny $d$}};
        \node[draw] (EE2) at (-7,7){{\tiny $e$}};
        
        \draw (AA1) -- (BB1) -- (CC1) -- (DD1) -- (EE1) (FF1) -- (CC2) -- (DD2) (BB2) -- (CC3) -- (FF2) (AA2) -- (BB3) -- (CC4) -- (DD3) -- (EE2);
        
        \draw (CC1) -- (FF1) (DD1) -- (CC2) -- (BB2) -- (AA2) (EE1) -- (DD2) -- (CC3) -- (BB3) (FF2) -- (CC4) (FF2) -- (CC4);
        
        \node (AA) at (-11,5.5){{\tiny $a$}};
        \node (BB) at (-10,5.5){{\tiny $b$}};
        \node (CC) at (-9,5.5){{\tiny $c$}};
        \node (DD) at (-8,5.5){{\tiny $d$}};
        \node (EE) at (-7,5.5){{\tiny $e$}};
        \node (FF) at (-9,6.5){{\tiny $f$}};
        
        \draw (AA) -- (BB) -- (CC) -- (DD) -- (EE) (CC) -- (FF); 
    \end{tikzpicture}};
    \end{tikzpicture}
    \caption{Hasse diagrams and colorings for all connected finite $\Gamma$-colored minuscule posets.}
    \label{FigFiniteMinusculeHasseDiagrams}
\end{figure} 



The connected infinite $\Gamma$-colored minuscule posets appearing in Theorem \ref{ThmMinusculeClassify}(ii) are the full heaps of R.M. Green.  These posets were the main objects of study in \cite{Gre} and were originally used by Green to construct representations of affine Lie algebras \cite{Gre1} and affine Weyl groups \cite{Gre2}.  Green classified all full heaps colored by Dynkin diagrams of affine Lie type in Theorem 6.6.2 of \cite{Gre}.  His doctoral student, Z.S. McGregor-Dorsey, showed in Theorem 4.7.1 of \cite{McD} that if a connected Dynkin diagram colors a full heap, then it must have affine Lie type.  Hence Green's classification lists all connected full heaps.  See \cite[Appx. C]{McD} for the Hasse diagrams of all connected full heaps and \cite[Appx. A]{McD} for the corresponding Dynkin diagrams of affine Lie type.
We translate between our setting and Green's in Section \ref{SectionHistory} and describe a connection between finite and infinite $\Gamma$-colored minuscule posets via the ``principal subheaps'' of full heaps.



\begin{theorem}\label{ThmMinusculeClassify}
Let $P$ be a $\Gamma$-colored poset.  Then $P$ is $\Gamma$-colored minuscule if and only if $P$ is a disjoint union of connected posets in which each is isomorphic to one of the following:
\begin{enumerate}[(i),nosep]
    \item A finite poset of type $A$ standard, $A$ exterior, $B$, $C$, $D$ standard, $D$ spin, $E_6$, or $E_7$, or
    \item An infinite poset that is one of the connected full heaps of Theorem 6.6.2 of \cite{Gre}.
\end{enumerate}
\noindent In this case $\Gamma$ is the disjoint union of the Dynkin diagrams coloring the connected components of $P$.  These Dynkin diagrams have finite Lie type for posets listed in (i) and affine Lie type for posets listed in (ii).
\end{theorem}


\begin{proof}
By \cite[Prop. 25]{dC-class}, a poset $(P,\Gamma,\kappa)$ is $\Gamma$-colored minuscule if and only if it decomposes into at most $|\Gamma|$ triples $(P_1,\Gamma_1,\kappa_1),\dots,(P_r,\Gamma_r,\kappa_r)$, where $P$ is the disjoint union of the connected posets $P_1,\dots,P_r$ and $\Gamma$ is the disjoint union of the connected Dynkin diagrams $\Gamma_1,\dots,\Gamma_r$ and $P_i$ is $\Gamma_i$-colored minuscule for all $1 \le i \le r$.  Hence the proof reduces to the connected case, so we now assume that $P$ is connected.


Suppose $P$ is a connected finite $\Gamma$-colored minuscule poset.  If $\Gamma$ is multiply laced, then $P$ has type $B$ or type $C$ by Theorem \ref{ThmMultLacedClassify}.  Now assume $\Gamma$ is simply laced.  If $P$ is a chain, then $P$ has type $A$ standard by Theorem \ref{ThmChainandSimplyLaced}.  Now assume $P$ is not a chain so that $\Gamma$ contains more than one color by EC.  Then $P$ is slant irreducible as a $\Gamma$-colored $d$-complete poset by Proposition \ref{PropChainorSlantIrreducible} and so Theorem \ref{ThmSimpLacedSlantIrred} shows $P$ has type $A$ exterior, $D$ standard, $D$ spin, $E_6$, or $E_7$.  Conversely, all of these types are $\Gamma$-colored minuscule.



Now suppose $P$ is a connected infinite $\Gamma$-colored minuscule poset.  Then both $P$ and the order dual $P^*$ are connected infinite $\Gamma$-colored $d$-complete posets.  Thus Theorem \ref{Thm22Str3} shows both $P$ and $P^*$ are filters of connected full heaps.  Hence for every $a \in \Gamma$, the set $P_a$ is unbounded above and below and is isomorphic as an uncolored poset to $\mathbb{Z}$. Thus G3 holds, and so $P$ is a full heap.  Conversely, if $P$ is a full heap, then it satisfies G3.  Hence UCB1 and LCB1 hold vacuously, and so $P$ is $\Gamma$-colored minuscule.


The statement about finite and affine Lie types holds by inspection; see \cite[\S 4.8]{Kac}.
\end{proof}




We now show how our work applies to Kac--Moody representation theory.  The algebraic terms used here are precisely defined in \cite{Kum}.  Let $\Gamma$ be a Dynkin diagram and set $n := |\Gamma|$ and $[n] := \{1,2,\dots,n\}$.  Fix a bijection $\nu : \Gamma \to [n]$ to number the nodes of $\Gamma$.  Under this numbering, the matrix $[\theta_{\nu(a),\nu(b)}]$ becomes a \emph{generalized Cartan matrix}.  For brevity, we will usually identify $a \in \Gamma$ with $\nu(a) \in [n]$; for example, we will write $[\theta_{ab}]$ for the generalized Cartan matrix with the ordering understood.  

\begin{remark}\label{RemSymmetrizable}
The generalized Cartan matrix $[\theta_{ab}]$ may or may not be \emph{symmetrizable}.  If $\Gamma$ colors either a $\Gamma$-colored $d$-complete or $\Gamma$-colored minuscule poset, then $[\theta_{ab}]$ is symmetrizable.  The $\Gamma$-colored $d$-complete case is handled by Proposition 27 of \cite{dC-class}, and the $\Gamma$-colored minuscule case follows from Theorem \ref{ThmMinusculeClassify} since $\Gamma$ must have connected components of finite or affine Lie type.
\end{remark}

The matrix $[\theta_{ab}]$ can be used to create the \emph{Kac--Moody algebra} $\mathfrak{g}$.  The \emph{Kac--Moody derived subalgebra} is $\mathfrak{g}' := [\mathfrak{g},\mathfrak{g}]$, where $[\cdot,\cdot]$ is the Lie bracket for $\mathfrak{g}$.  This subalgebra can be defined with generators $\{x_a,y_a,h_a\}_{a \in \Gamma}$ subject to the following relations:
\begin{itemize}[nosep]
    \item [] (XX) \  $\underbrace{[x_a,[x_a,\dots,[x_a}_{\text{$1 - \theta_{ba}$ times}},x_b] \dots ]] = 0$ for all $a,b \in \Gamma$ such that $a \ne b$,
    \item [] (YY) \ $\underbrace{[y_a,[y_a,\dots,[y_a}_{\text{$1 - \theta_{ba}$ times}},y_b] \dots]] = 0$ for all $a,b \in \Gamma$ such that $a \ne b$,
    \item [] (HH) \ $[h_b,h_a] = 0$ for all $a,b \in \Gamma$,
    \item [] (HX) \ $[h_b,x_a] = \theta_{ab} x_a$ for all $a,b \in \Gamma$,
    \item [] (HY) \ $[h_b,y_a] = -\theta_{ab} y_a$ for all $a,b \in \Gamma$, and
    \item [] (XY) \ $[x_a,y_b] = \delta_{ab} h_a$ for all $a,b \in \Gamma$, where $\delta_{ab}$ is the Kronecker delta.
\end{itemize} 




Let $P$ be a locally finite $\Gamma$-colored poset.  The \emph{splits} of $P$ are the pairs $(F,I)$ where $F$ is a filter of $P$ and $I := P - F$ is its corresponding ideal.  Let $\mathcal{FI}(P)$ denote the set of all splits of $P$ and let $V := \langle \mathcal{FI}(P) \rangle$ be the complex vector space with basis $\mathcal{FI}(P)$.  The basis vector corresponding to $(F,I) \in \mathcal{FI}(P)$ is denoted $\langle F,I \rangle$.  Let $a \in \Gamma$ and $(F,I) \in \mathcal{FI}(P)$.  If there are finitely many elements of color $a$ that are minimal in $F$ (respectively maximal in $I$), then define $X_a.\langle F,I \rangle := \sum \langle F-\{x\},I \cup \{x\} \rangle$ (respectively $Y_a.\langle F,I \rangle := \sum \langle F \cup \{x\},I-\{x\}\rangle$), where the sum is taken over those minimal (respectively maximal) elements of $F$ (respectively $I$) of color $a$.  If these sums exist for all $b \in \Gamma$ and $(F,I) \in \mathcal{FI}(P)$, then extend these functions linearly to $V$ and call them the \emph{color raising} (respectively \emph{lowering}) \emph{operators} on $V$.  Note that the property EC is sufficient to guarantee these operators exist, since then all sums defining these operators will have at most one term.

We made the following definitions in \cite{Str,Unify}.

\begin{definition}\label{DefPMinuscule}
Suppose $P$ is a locally finite $\Gamma$-colored poset and the operators $\{X_a,Y_a\}_{a \in \Gamma}$ are defined on $V$.
\begin{enumerate}[(a),nosep]
    \item We say $\mathcal{FI}(P)$ \emph{carries a representation of $\mathfrak{g}'$} if there exist diagonal operators $\{H_a\}_{a \in \Gamma}$ (with respect to the basis of splits) such that the operators $\{X_a,Y_a,H_a\}_{a \in \Gamma}$ satisfy the relations XX, YY, HX, HY, and XY for $\mathfrak{g}'$ under the commutator $[A,B] := AB - BA$ on $\text{End}(V)$.
    \item A representation of $\mathfrak{g}'$ carried by $\mathcal{FI}(P)$ is \emph{$P$-minuscule} if the eigenvalues for the diagonal operators $\{H_a\}_{a \in \Gamma}$ on the basis of splits are in the set $\{-1,0,1\}$.
\end{enumerate}
\end{definition}

\noindent There was no need to require the relation HH in Part (a) of Definition \ref{DefPMinuscule}; this relation holds automatically since the operators $\{H_a\}_{a \in \Gamma}$ are diagonal on the basis of splits.


One of the main results from \cite{Unify} stated that a $\Gamma$-colored poset $P$ is $\Gamma$-colored minuscule if and only if the lattice $\mathcal{FI}(P)$ carries a $P$-minuscule representation of $\mathfrak{g}'$.
This result was obtained in the simply laced case as Theorem 38(b) of \cite{Unify} and in the general case as Theorem 6.1.1(b) of \cite{Str}.  Applying Theorem \ref{ThmMinusculeClassify} obtains the classification of $P$-minuscule representations.

\begin{theorem}\label{ThmPMinusculeClassify}
Let $P$ be a $\Gamma$-colored poset.  Then $\mathcal{FI}(P)$ carries a $P$-minuscule representation of $\mathfrak{g}'$ if and only if $P$ is a disjoint union of connected posets in which each is isomorphic to one of the following:
\begin{enumerate}[(i),nosep]
    \item A finite poset of type $A$ standard, $A$ exterior, $B$, $C$, $D$ standard, $D$ spin, $E_6$, or $E_7$, or
    \item An infinite poset that is one of the connected full heaps of Theorem 6.6.2 of \cite{Gre}.  $\qed$
\end{enumerate}
\end{theorem}

\noindent For a $P$-minuscule representation of $\mathfrak{g}'$, the diagonal operators $\{H_a\}_{a \in \Gamma}$ are uniquely determined by the rule
$$ H_b.\langle F,I \rangle = \left\{ \begin{array}{rl} 
-\langle F,I \rangle & \text{if $b$ is the color of a minimal element of $F$} \\
\langle F,I \rangle & \text{if $b$ is the color of a maximal element of $I$} \\
0 & \text{otherwise}
\end{array} \right. $$
for every $b \in \Gamma$ and $(F,I) \in \mathcal{FI}(P)$.  See Proposition 33 and Theorem 35 of \cite{Unify} for the simply laced case and Theorem 5.4.2 of \cite{Str} for the general case.


\section{Minuscule posets, Weyl group elements, and coroots}\label{SectionHistory}

In Theorem \ref{ThmColoredCoroots} we realize connected finite $\Gamma$-colored minuscule posets as posets of colored coroots. This result recreates the 1984 theorem of Proctor \cite[Thm. 11]{BLPP} that produced the first appearance of colored minuscule posets, but in our axiomatic setting.  We also discuss other appearances of $\Gamma$-colored minuscule posets.  The algebraic objects we are primarily concerned with are semisimple Lie algebras and their representations, and so \cite{Hum} may be consulted for most definitions and basic facts.



Let $\mathfrak{g}$ be a \textit{semisimple Lie algebra} of \textit{rank} $n$.
Finite dimensional \textit{irreducible representations} of $\mathfrak{g}$ are parameterized by the \textit{dominant integral weights} via their \textit{highest weights}.  
The \textit{Weyl group} $W$ acts on the set of \textit{weights} of each representation.  
If $V(\lambda)$ is the irreducible representation of $\mathfrak{g}$ with highest weight $\lambda$, then it is a \textit{minuscule representation} if the Weyl group action on weights is transitive.  
Such highest weights $\lambda$ are called \textit{minuscule weights}; they exist in finite Lie types $A_n$ for $n > 0$, $B_n$ for $n > 1$, $C_n$ for $n > 2$, $D_n$ for $n > 3$, and $E_6$ and $E_7$.  
(Restrictions on $n$ prevent redundant types.)
Each minuscule weight is one of the \textit{fundamental weights} $\{\omega_1,\dots,\omega_n\}$; the list of minuscule weights is given in \cite[Ch. VIII, \S 7.3]{Bou}.

The Weyl group is generated by the \textit{simple reflections} in $S = \{s_1,\dots,s_n\}$.  
If $J \subseteq S$, then it generates the \textit{parabolic subgroup} $W_J$ of $W$.  
The left cosets in the quotient $W / W_J$ each have a unique minimal \textit{length} representative; let $W^J$ denote the set of these elements.  
Here $W^J$ inherits the \textit{Bruhat order} from $W$ (see \cite{BjBr}).  
The Bruhat poset $W^J$ was developed to study the corresponding flag manifold $G / P_J$.
Denote $\langle j \rangle := S - \{s_j\}$.
By 1978, C.S. Seshadri had obtained his standard monomial bases \cite{Ses} for the minuscule flag manifolds $G/P_{\langle j \rangle}$ by working with the Bruhat poset $W^{\langle j \rangle}$ when $\omega_j$ is minuscule.

The weights of any representation of $\mathfrak{g}$ are ordered by the \textit{simple roots} $\{\alpha_1,\dots,\alpha_n\}$ via the rule $\nu \le \mu$ if $\mu - \nu$ is a non-negative integral sum of simple roots.
Let $\omega_j$ for $1 \le j \le n$ be a minuscule weight and let $V(\omega_j)$ be the corresponding minuscule representation.
Using a process (later called the numbers game of Mozes \cite{Moz}) to simultaneously generate the weight diagram of $V(\omega_j)$ and minimal length coset representatives in $W^{\langle j \rangle}$, Proctor proved \cite[\S \S 3--4]{BLPP} that the weight diagram of $V(\omega_j)$ is dual isomorphic to $W^{\langle j \rangle}$ and that these structures are distributive lattices.  
Let $P$ be the poset of join irreducible elements of the weight diagram of $V(\omega_j)$.
Then this weight diagram is isomorphic to the distributive lattice $J(P)$ of order ideals of $P$, ordered by inclusion.  
The set $\mathcal{FI}(P)$ of splits introduced in Section \ref{SectionMainResults} is just a reformulation of $J(P)$.
If $\mathfrak{g}$ has finite Lie type $\mathcal{L}_n$ for $\mathcal{L} \in \{A,B,C,D,E\}$, then we denote this uncolored poset $P$ by $\ell_n(j)$.

Posets that arise this way are thus indexed by minuscule weights; the full list
is $a_n(1),a_n(2),\dots,a_n(n)$ for $n > 0$, $b_n(n)$ for $n > 1$, $c_n(1)$ for $n > 2$, $d_n(1),d_n(n-1)$, and $d_n(n)$ for $n > 3$, $e_6(1)$, $e_6(5)$, and $e_7(6)$, with fundamental weights corresponding to the node numbering given in \cite[\S 4.8]{Kac}. 
These posets are the irreducible \textit{minuscule posets} of \cite[\S 4]{BLPP}.
We remark that some minuscule posets are isomorphic (e.g. $a_n(j) \cong a_n(n+1-j)$ for $1 \le j \le n$), even across Lie types (e.g. $b_{n-1}(n-1) \cong d_n(n) \cong d_n(n-1)$ for $n > 3$), but are listed separately to emphasize the distinct minuscule representations from which they arise.

We introduce an analogous notation for the connected finite $\Gamma$-colored minuscule posets of Theorem \ref{ThmMinusculeClassify}(i) that is also indexed by minuscule weights.
Fix such a poset $P$ with coloring $\kappa : P \to \Gamma$, and let $n := |\Gamma|$.  
If $P$ has type $\mathcal{L}$ for $\mathcal{L} \in \{A,B,C,D,E\}$ (possibly with standard, exterior, or spin variant), then $\Gamma$ has finite Lie type $\mathcal{L}_n$.
Let $\text{Aut}(\Gamma)$ be the group of diagram automorphisms of $\Gamma$.  
This group has order one for $A_1$, two for $A_n$ for $n > 1$, one for $B_n$ for $n > 1$, one for $C_n$ for $n > 2$, six for $D_4$, two for $D_n$ for $n > 4$, two for $E_6$, and one for $E_7$.  
Fix $\varphi \in \text{Aut}(\Gamma)$ and suppose the maximal element of $P$ has color $a$. 
Using the numbering $\nu : \Gamma \to [n]$ from \cite[\S 4.8]{Kac}, we denote by $\ell_n^\kappa(\nu \varphi(a))$ the $\Gamma$-colored minuscule poset obtained from $P$ by replacing all colors with the new set of colors $[n]$ via the bijection $\nu \varphi : \Gamma \to [n]$.

Repeating for each connected finite $\Gamma$-colored minuscule poset of type $\mathcal{L}$ and each $\varphi \in \text{Aut}(\Gamma)$ produces exactly one poset $\ell_n^\kappa(j)$ up to $\Gamma$-colored poset isomorphism for each minuscule weight $\omega_j$ in type $\mathcal{L}_n$.  The maximal element of $\ell_n^\kappa(j)$ has color $j$.
This notation for the $\Gamma$-colored case is compatible with the uncolored case in the sense that the uncolored poset $\ell_n(j)$ is the underlying poset for $\ell_n^\kappa(j)$.



\begin{figure}[t!]
    \centering
    \begin{tikzpicture}[scale=.65]
    
    \node[draw] (A1) at (10,10){1};
    \node[draw] (B1) at (11,9){2};
    \node[draw] (C1) at (12,8){3};
    \node[draw] (D1) at (13,7){4};
    \draw (A1) -- (B1) -- (C1) -- (D1);
    
    \node (11) at (10,5.5){1};
    \node (12) at (11,5.5){2};
    \node (13) at (12,5.5){3};
    \node (14) at (13,5.5){4};
    \draw (11) -- (12) -- (13) -- (14);
    
    \node[draw] (A2) at (16,9){1};
    \node[draw] (B21) at (17,10){2};
    \node[draw] (B22) at (17,8){2};
    \node[draw] (C21) at (18,9){3};
    \node[draw] (C22) at (18,7){3};
    \node[draw] (D2) at (19,8){4};
    \draw (A2) -- (B21) -- (C21) -- (B22) -- (A2) (C21) -- (D2) -- (C22) -- (B22);
    
    \node (21) at (16,5.5){1};
    \node (22) at (17,5.5){2};
    \node (23) at (18,5.5){3};
    \node (24) at (19,5.5){4};
    \draw (21) -- (22) -- (23) -- (24);
    
    \node[draw] (A3) at (22,9){4};
    \node[draw] (B31) at (23,10){3};
    \node[draw] (B32) at (23,8){3};
    \node[draw] (C31) at (24,9){2};
    \node[draw] (C32) at (24,7){2};
    \node[draw] (D3) at (25,8){1};
    \draw (A3) -- (B31) -- (C31) -- (B32) -- (A3) (C31) -- (D3) -- (C32) -- (B32);
    
    \node (31) at (22,5.5){4};
    \node (32) at (23,5.5){3};
    \node (33) at (24,5.5){2};
    \node (34) at (25,5.5){1};
    \draw (31) -- (32) -- (33) -- (34);
    
    \node[draw] (A4) at (28,10){4};
    \node[draw] (B4) at (29,9){3};
    \node[draw] (C4) at (30,8){2};
    \node[draw] (D4) at (31,7){1};
    \draw (A4) -- (B4) -- (C4) -- (D4);
    
    \node (41) at (28,5.5){4};
    \node (42) at (29,5.5){3};
    \node (43) at (30,5.5){2};
    \node (44) at (31,5.5){1};
    \draw (41) -- (42) -- (43) -- (44);
    
    
    \end{tikzpicture}
    \caption{From left to right: The $\Gamma$-colored minuscule posets $a_4^\kappa(1)$, $a_4^\kappa(2)$, $a_4^\kappa(3)$, and $a_4^\kappa(4)$.}
    \label{FigA4Example}
\end{figure}
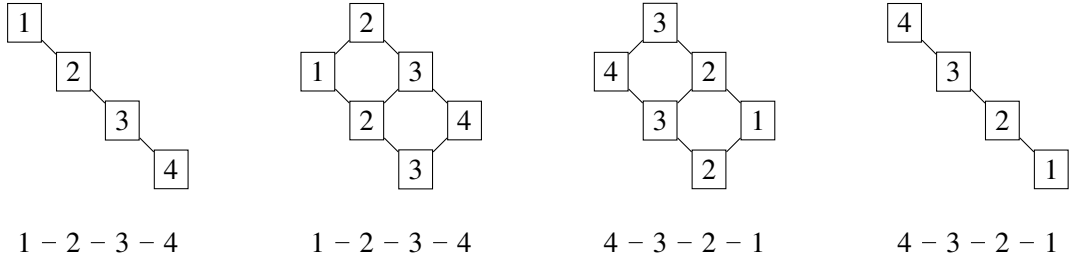

See Figure \ref{FigA4Example} for $a_4^\kappa(1),a_4^\kappa(2),a_4^\kappa(3)$, and $a_4^\kappa(4)$ in type $A_4$.
In this example, both $a_4^\kappa(1)$ and $a_4^\kappa(4)$ come from the type $A$ standard $\Gamma$-colored minuscule poset with four colors and both $a_4^\kappa(2)$ and $a_4^\kappa(3)$ come from the type $A$ exterior $\Gamma$-colored minuscule poset with four colors.  
Both $a_4^\kappa(1)$ and $a_4^\kappa(2)$ come from the identity automorphism on $\Gamma$ and both $a_4^\kappa(3)$ and $a_4^\kappa(4)$ come from the nonidentity automorphism on $\Gamma$.
Sometimes a redundant copy of a poset will be produced, as can be seen in type $A_5$ with $a_5^\kappa(3)$ arising 
twice.





Proctor's results discussed above first appeared in his 1980 thesis before appearing in Sections 3 and 4 of \cite{BLPP}.  
After reading this thesis, Robert Steinberg suggested exploring the relationship between the uncolored minuscule posets $\ell_n(j)$ of join irreducibles and certain sets of roots, at least in types $A$, $D$, and $E$.
In Theorem 11 of \cite{BLPP}, Proctor gave a general realization of each poset on his list with certain sets of \textit{coroots} in the corresponding finite Lie type.
Theorem 11 also contained a coloring of these coroots.
We recover these realizations for connected finite $\Gamma$-colored minuscule posets in Theorem \ref{ThmColoredCoroots} below.

Let $\left(\Phi^\vee\right)^+$ and $\left(\Phi^\vee\right)^-$ be the \textit{positive} and \textit{negative} coroots in type $\mathcal{L}_n$, respectively, and $\Phi^\vee := \left(\Phi^\vee\right)^+ \cup \left(\Phi^\vee\right)^-$.
The standard order on $\Phi^\vee$ is analogous to the root order described above.
Let $P$ be the connected finite $\Gamma$-colored minuscule poset $\ell_n^\kappa(j)$ of finite type $\mathcal{L}_n$.
Let $F$ be a nonempty filter of $P$.
Then $F$ is a $\Gamma'$-colored $d$-complete poset, where $\Gamma'$ is the (connected) subdiagram of $\Gamma$ with colors in $\kappa(F)$.  This subdiagram $\Gamma'$ corresponds to a coroot subsystem of $\Phi^\vee$ and a parabolic subgroup of $W$; we consider all Lie objects for $\Gamma'$ (e.g. coroots, Weyl group elements) to have type $\mathcal{L}_n$.  In other words, we are considering $F$ to be a $\Gamma$-colored $d$-complete poset even if its coloring is not surjective.

 In \cite{Ste}, Stembridge did not require dominant minuscule heaps to be colored surjectively outside of his classification in Section 4; we use only his Sections 3 and 5 below.  Thus as a $\Gamma$-colored $d$-complete poset, the filter $F$ from the previous paragraph is a dominant minuscule heap by Theorem \ref{Thm9Str3}.  As in \cite[\S 3]{Ste}, there is a dominant integral weight $\lambda$ and $\lambda$-minuscule \cite{Car} Weyl group element $w$ such that $F$ is the \textit{heap} of $w$.  All \textit{reduced expressions} of $w$ are produced from linear extensions of $F$ by recording the order of colors appearing in the linear extension and producing a word in $W$ as a product of the corresponding simple reflections.  Hence the length of $w$ is $l(w) = |F|$.  Since $W$ also acts on coroots, let $\Phi^\vee(w) := \left\{ \alpha^\vee \in \left(\Phi^\vee\right)^+ \ | \ w\alpha^\vee \in \left(\Phi^\vee\right)^-\right\}$ be the \textit{inversion set} of $w$; i.e. the positive coroots that become negative under the action of $w$.  

For use in the proof of Theorem \ref{ThmColoredCoroots}, we state the following lemma in the context just described.  We note that it holds more generally for any connected finite $\Gamma$-colored $d$-complete poset, which must have symmetrizable Kac--Moody type by Remark \ref{RemSymmetrizable}.  Recall by Lemma \ref{LemUniqueMaxElt} that $P$ has a unique maximal element.



\begin{lemma}\label{LemdC-Coroots}
Let $P$ be the connected finite $\Gamma$-colored minuscule poset $\ell_n^\kappa(j)$.  Let $F$ be a nonempty filter of $P$ and let $w$ be the dominant $\lambda$-minuscule Weyl group element for which $F$ is its heap. 
\begin{enumerate}[(a),nosep]
\item The element $w$ is in $W^{\langle j \rangle}$.
\item \hspace{0in} The filter $F$ is dual isomorphic as an uncolored poset to $\Phi^\vee(w)$.  
\item Let $\Phi_j^\vee$ be the filter of $\left(\Phi^\vee\right)^+$ generated by $\alpha_j^\vee$.
Then $\Phi^\vee(w)$ is an order ideal of $\Phi_j^\vee$.
\end{enumerate} 
\end{lemma}

\begin{proof}
Every linear extension of $F$ ends with the unique maximal element of $P$ colored $j$, so every reduced expression of $w$ ends in $s_j$.  Using this fact, Lemma 2.4.3 and Corollary 2.4.5(i) of \cite{BjBr} show that $w \in W^{\langle j \rangle}$, proving (a). 
Theorem 5.5 of \cite{Ste}, which holds for any symmetrizable Kac--Moody type, implies (b).  For (c), note using (b) that the unique maximal element of $F$ corresponds to a unique minimal element of $\Phi^\vee(w)$.  Since every reduced expression for $w$ ends in $s_j$, Corollary C of \cite[\S 10.2]{Hum} shows that $\alpha_j^\vee \in \Phi^\vee(w)$.  Hence $\alpha_j^\vee$ must be the unique minimal element of $\Phi^\vee(w)$ since it is minimal in $\left(\Phi^\vee\right)^+$.  Since $\Phi^\vee(w)$ is a convex subposet of coroots \cite[Rem. 5.6(a)]{Ste}, we see that $\Phi^\vee(w)$ is an ideal of $\Phi_j^\vee$. 
\end{proof}

We now obtain the realization of $P = \ell_n^\kappa(j)$ as a poset of colored coroots.  Recall by Lemma \ref{LemUniqueMaxElt} that $P$ has a unique minimal element.

\begin{theorem}\label{ThmColoredCoroots}
Let $P$ be the connected finite $\Gamma$-colored minuscule poset $\ell_n^\kappa(j)$.  For every $x \in P$, we define $F_x$ to be the filter generated by $x$ and $w_x$ to be the dominant $\lambda$-minuscule Weyl group element for which $F_x$ is its heap.  Let $\Phi_j^\vee$ be the filter of $\left(\Phi^\vee\right)^+$ generated by $\alpha_j^\vee$; its unique maximal element is the highest coroot in $\Phi^\vee$. 
\begin{enumerate}[(a),nosep]
    \item For all $x \in P$, the filter $F_x$ is dual isomorphic as an uncolored poset to $\Phi^\vee(w_x)$.  This poset $\Phi^\vee(w_x)$ of coroots is an ideal of $\Phi_j^\vee$ generated by a single element, which we denote $\gamma_x^\vee$.
    \item If $x$ is the unique minimal element of $P$, then $\Phi^\vee(w_x) = \Phi_j^\vee$ and $w_x$ is the longest element $w_0^{\langle j \rangle}$ of $W^{\langle j \rangle}$.
    \item The map $\psi : P \to \Phi_j^\vee$ defined by $\psi(x) := \gamma_x^\vee$ for all $x \in P$ is a dual isomorphism of uncolored posets.
    \item The map $\psi$ induces a coloring $\kappa_\psi : \Phi_j^\vee \to \Gamma$ of $\Phi_j^\vee$ defined by $\kappa_\psi\left(\gamma_x^\vee\right) := \kappa(x)$.  Under this coloring, the posets $P$ and $\Phi_j^\vee$ are dual isomorphic as $\Gamma$-colored posets.  Hence $\Phi_j^\vee$ is $\Gamma$-colored minuscule.
\end{enumerate}
\end{theorem}

\begin{proof}
Let $x \in P$.  We apply Lemma \ref{LemdC-Coroots} to the filter $F_x$ and Weyl group element $w_x$.  Part (b) shows that $F_x$ is dual isomorphic as an uncolored poset to $\Phi^\vee(w_x)$ and Part (c) shows that $\Phi^\vee(w_x)$ is an ideal of $\Phi_j^\vee$.  Since $F_x$ has a unique minimal element, this shows $\Phi^\vee(w_x)$ has a unique maximal element, proving (a).  As in the statement, we denote this element by $\gamma_x^\vee$.

Now suppose $x$ is the unique minimal element of $P$, so that $F_x = P$.  Since the unique minimal element of $\Phi^\vee(w_x)$ has \textit{height} 1 and since $P$ and $\Phi^\vee(w_x)$ are dual isomorphic, the largest coroot height appearing in $\Phi^\vee(w_x)$ is given by the number of ranks of $P$.  Note that $P$ is one of the posets displayed in Figure \ref{FigFiniteMinusculeHasseDiagrams}.  By inspecting these Hasse diagrams, this largest coroot height is $n$, $2n-1$, $2n-1$, $2n-3$, $11$, and $17$ in types $A_n$, $B_n$, $C_n$, $D_n$, $E_6$, and $E_7$, respectively.  These are the heights of the unique highest coroot in each respective type, and so $\Phi^\vee(w_x)$ contains the highest coroot.  That is, the highest coroot is $\gamma_x^\vee$ and so $\Phi^\vee(w_x) = \Phi_j^\vee$.

We note that $w_x \in W^{\langle j \rangle}$ by Lemma \ref{LemdC-Coroots}(a).  An alternate characterization of $W^{\langle j \rangle}$ is the set of Weyl group elements whose actions on the coroots in $\left(\Phi^\vee\right)^+ - \Phi_j^\vee$ remain positive; see \cite[Exer. 1.3.E]{Kum}.  Thus $\Phi^\vee\left(w_0^{\langle j \rangle}\right) \subseteq \Phi_j^\vee$ for the longest element $w_0^{\langle j \rangle}$ of $W^{\langle j \rangle}$.  Since Weyl group length satisfies $l(u) = |\Phi^\vee(u)|$ for every $u \in W$, it follows that $l\left(w_0^{\langle j \rangle}\right) \le |\Phi_j^\vee|$.  Hence $l(w_x) = \left|\Phi^\vee(w_x)\right| = \left|\Phi_j^\vee\right|$ shows that $w_x$ has maximum possible length in $W^{\langle j \rangle}$, i.e. $w_x = w_0^{\langle j \rangle}$.  This finishes the proof of (b).


The second paragraph of this proof shows that $P$ is dual isomorphic to $\Phi_j^\vee$.  Hence $|P| = \left| \Phi_j^\vee \right|$, and so to show $\psi$ is a bijection it suffices to show it is injective.  Suppose that $x,y \in P$ with $\psi(x) = \psi(y)$.  Since the ideals $\Phi^\vee(w_x)$ and $\Phi^\vee(w_y)$ of $\Phi_j^\vee$ are respectively generated by $\gamma_x^\vee$ and $\gamma_y^\vee$, which are equal by assumption, this shows that $\Phi^\vee(w_x) = \Phi^\vee(w_y)$.  Hence $w_x = w_y$ since distinct Weyl group elements have distinct inversion sets.  Since reduced expressions for $w_x$ and $w_y$ are respectively produced from linear extensions of the filters $F_x$ and $F_y$, this shows $F_x = F_y$.  Thus $x = y$ and so $\psi$ is injective.

Now let $x$ be any element of $P$; we produce an explicit realization of $\Phi^\vee(w_x)$.  Suppose $|F_x| = f$ and let $x_f := x \to \cdots \to x_1$ be any linear extension of $F_x$.  Set $i_k := \kappa(x_k)$ for $1 \le k \le f$.  Then $s_{i_f} \cdots s_{i_1}$ is a reduced expression for $w_x$ and $\Phi^\vee(w_x)$ is given by
\begin{align}\label{EqCoroots}
    \alpha_{i_1}^\vee, \ \ \ s_{i_1}\left(\alpha_{i_2}^\vee \right), \ \ \ \dots \ \ \ , s_{i_1} \cdots s_{i_{f-2}}\left(\alpha_{i_{f-1}}^\vee \right), \ \ \ s_{i_1} \cdots s_{i_{f-1}} \left(\alpha_{i_f}^\vee \right);
\end{align}
for example, see \cite[Exer. 5.6.1]{Hum2}.  If $f = 1$, then $x$ is the unique maximal element of $P$ and so $\gamma_x^\vee = \alpha_j^\vee = \alpha_{i_1}^\vee$.  If  $f > 1$, then this realization can be repeated with the filter $F' := F - \{x\}$ and its corresponding Weyl group element $w' := s_{i_{f-1}} \cdots s_{i_1}$, producing $\Phi^\vee\left(w'\right)$ as the first $f-1$ coroots in sequence (\ref{EqCoroots}).  
Since $\Phi^\vee\left(w'\right)$ is an ideal of $\Phi_j^\vee$ by Lemma \ref{LemdC-Coroots}(c), we have $\Phi^\vee(w') = \Phi^\vee(w_x) - \left\{ \gamma_x^\vee \right\}$.  So $\gamma_x^\vee$ is the final coroot in (\ref{EqCoroots}).

Suppose that $x \le y$ in $P$.  Form a linear extension of $F_y$ and extend it to a linear extension of $F_x$.  By the previous paragraph, this linear extension for $F_x$ produces a coroot sequence as in (\ref{EqCoroots}) culminating in $\gamma_x^\vee$.  The coroot in position $|F_y|$ of this sequence is $\gamma_y^\vee$ (also by the preceding paragraph), so $\gamma_y^\vee \in \Phi^\vee(w_x)$.  Since $\gamma_x^\vee$ is the unique maximal element of $\Phi^\vee(w_x)$, we see that $\gamma_x^\vee \ge \gamma_y^\vee$.  Hence $\psi$ is an order reversing bijection.  Since we know that $P$ is dual isomorphic to $\Phi_j^\vee$, it follows that $\psi$ must be a dual isomorphism; i.e. that $\gamma_x^\vee \ge \gamma_y^\vee$ implies $x \le y$ as well.  This proves (c).

Let $\text{id}_\Gamma : \Gamma \to \Gamma$ be the identity automorphism.  Then $\kappa_\psi \psi = \text{id}_\Gamma \kappa$ by definition, so $P$ and $\Phi_j^\vee$ are dual isomorphic as $\Gamma$-colored posets, proving (d).
\end{proof}

\begin{remark}\label{RemExtensionByCoroots}
\begin{enumerate}[(a),nosep]
    \item Let $P$ be a connected finite $\Gamma$-colored $d$-complete poset and suppose $\Gamma$ has finite Lie type $\mathcal{L}_n$.  Suppose that $j$ is the color of the unique maximal element of $P$.  Lemma \ref{LemdC-Coroots} and Theorem \ref{ThmColoredCoroots} can be viewed as showing that the downward color extension process used in Section \ref{SectionExtending} to produce new $\Gamma$-colored $d$-complete posets corresponds to lengthening elements of $W^{\langle j \rangle}$ by multiplying by simple reflections on the left and to growing ideals of $\Phi_j^\vee$ upwardly.  These parallel extensions respectively produce $\ell_n^\kappa(j)$, $w_0^{\langle j \rangle}$, and $\Phi_j^\vee$ when the process terminates with a $\Gamma$-colored minuscule poset.
    
    \item 
    The coloring of $\Phi_j^\vee$ given in Theorem \ref{ThmColoredCoroots}(d) can be seen to be the coloring given by Proctor in \cite[Thm. 11]{BLPP} 
    after adjusting for differing conventions.  Proctor was not using the axiomatically defined poset $\ell_n^\kappa(j)$, so he instead produced the Weyl group words and coroot actions appearing in Section 11 of \cite{BLPP} using the numbers game.  These colored sets of coroots were the first appearance of colored minuscule posets, as noted in the top left corner of Table \ref{RepClass} displayed in the introduction.
\end{enumerate}
\end{remark}

\begin{figure}[t!]
    \centering
    \begin{tikzpicture}[scale=1.25]
        \node(1) at (0,0){{\footnotesize $\alpha_1^\vee$}};
        \node(2) at (2,0){{\footnotesize $\alpha_2^\vee$}};
        \node(3) at (4,0){{\footnotesize $\alpha_3^\vee$}};
        \node(4) at (6,0){{\footnotesize $\alpha_4^\vee$}};
        \node(12) at (1,1){{\footnotesize $\alpha_1^\vee + \alpha_2^\vee$}};
        \node(23) at (3,1){{\footnotesize $\alpha_2^\vee + \alpha_3^\vee$}};
        \node(34) at (5,1){{\footnotesize $\alpha_3^\vee + \alpha_4^\vee$}};
        \node(123) at (2,2){{\footnotesize $\alpha_1^\vee + \alpha_2^\vee + \alpha_3^\vee$}};
        \node(234) at (4,2){{\footnotesize $\alpha_2^\vee + \alpha_3^\vee + \alpha_4^\vee$}};
        \node(1234) at (3,3){{\footnotesize $\alpha_1^\vee + \alpha_2^\vee + \alpha_3^\vee + \alpha_4^\vee$}};
        
        \node(2a) at (2,-.3){{\footnotesize $2$}};
        \node(3a) at (3,0.7){{\footnotesize $3$}};
        \node(1a) at (1,0.7){{\footnotesize $1$}};
        \node(4a) at (4,1.7){{\footnotesize $4$}};
        \node(2b) at (2,1.7){{\footnotesize $2$}};
        \node(3b) at (3,2.7){{\footnotesize $3$}};
        
        \node(gamma1) at (2,0.3){{\footnotesize $\gamma_u^\vee$}};
        \node(gamma2) at (3,1.3){{\footnotesize $\gamma_w^\vee$}};
        \node(gamma3) at (1,1.3){{\footnotesize $\gamma_v^\vee$}};
        \node(gamma4) at (4,2.3){{\footnotesize $\gamma_y^\vee$}};
        \node(gamma5) at (2,2.3){{\footnotesize $\gamma_x^\vee$}};
        \node(gamma6) at (3,3.3){{\footnotesize $\gamma_z^\vee$}};
        
        \draw (1) -- (12) -- (2) -- (23) -- (3) -- (34) -- (4) (12) -- (123) -- (23) -- (234) -- (34) (123) -- (1234) -- (234);
        
        \draw[dashed] (2,-1.25) -- (-0.25,1) -- (3,4.25) -- (5.25,2) -- (2,-1.25);
        
        \node(c1)[draw] at (-3,0){{\footnotesize $z_3$}};
        \node(d1)[draw] at (-2,1){{\footnotesize $y_4$}};
        \node(b1)[draw] at (-4,1){{\footnotesize $x_2$}};
        \node(a1)[draw] at (-5,2){{\footnotesize $v_1$}};
        \node(b2)[draw] at (-4,3){{\footnotesize $u_2$}};
        \node(c2)[draw] at (-3,2){{\footnotesize $w_3$}};
        
        \node(1) at (-5,-1){{\footnotesize $1$}};
        \node(2) at (-4,-1){{\footnotesize $2$}};
        \node(3) at (-3,-1){{\footnotesize $3$}};
        \node(4) at (-2,-1){{\footnotesize $4$}};
        
        \draw (1) -- (2) -- (3) -- (4);
        
        \draw (b1) -- (c1) -- (d1) -- (c2) -- (b2) -- (a1) -- (b1) -- (c2);
    \end{tikzpicture}
    \caption{The $\Gamma$-colored minuscule poset $a_4^\kappa(2)$ is displayed on the left, with colors displayed as subscripts.  Positive coroots of type $A_4$ are displayed on the right, with the set $\Phi_2^\vee$ in the dashed box.}
    \label{FigA4Coroots}
\end{figure}
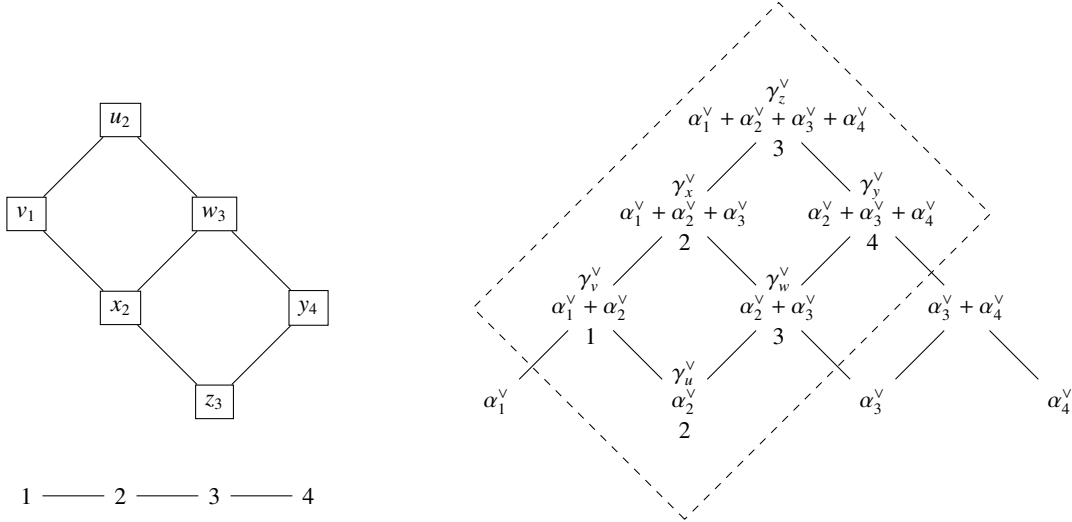

See Figure \ref{FigA4Coroots} for an example consisting of the $\Gamma$-colored minuscule poset $a_4^\kappa(2)$ and the dual isomorphic poset $\Phi_2^\vee$ in type $A_4$.
As in Theorem \ref{ThmColoredCoroots}, each element $s \in a_4^\kappa(2)$ maps to $\gamma_s^\vee$ under the dual isomorphism $\psi$; 
a linear extension of $F_s$ gives a reduced expression for $w_s$, which recovers $\gamma_s^\vee$ as the last coroot of the sequence (\ref{EqCoroots}).  For example, the linear extension $z \to y \to x \to w \to v \to u$ of $F_z$ produces the Weyl group element $w_z = s_3s_4s_2s_3s_1s_2$, which is the longest element $w_0^{\langle 2 \rangle}$ in $W^{\langle 2 \rangle}$.
Then $\gamma_z^\vee = s_2s_1s_3s_2s_4\left(\alpha_3^\vee\right) = \alpha_1^\vee + \alpha_2^\vee + \alpha_3^\vee + \alpha_4^\vee$.
Colors for $\Phi_2^\vee$ given by $\kappa_\psi$ are displayed under each coroot in the figure.  

If one is only interested in uncolored posets, then there is no need to dualize since uncolored minuscule posets are self-dual.  
Likewise, in the uncolored minuscule poset context, 
one may identify the poset $\ell_n(j)$ with the filter $\Phi_j$ of roots generated by $\alpha_j$ in type $\mathcal{L}_n$; for example, see \cite[Thm. 8.3.10]{Gre}.
But in the case of roots, the analogous inversion set $\Phi(w_s)$ for $s \in \ell_n^\kappa(j)$ is not necessarily an ideal of $\Phi_j$ in multiply laced types, and so attempting to color $\Phi_j$ as was done in Theorem \ref{ThmColoredCoroots} to produce a $\Gamma$-colored minuscule poset fails.  One may use the poset $b_2^\kappa(2)$ and roots $\Phi_2$ in type $B_2$ as an example.


Connected finite $\Gamma$-colored minuscule posets also appear as the \textit{principal subheaps} of Green's full heaps (see  \cite[Def. 5.5.3]{Gre}). 
His notation for full heaps in Theorem 6.6.2 of \cite{Gre}, which we cite in our classification Theorem \ref{ThmMinusculeClassify}(ii), uses this association.  
Let $\varepsilon : \text{FH}(\Gamma(j)) \to \Gamma$ be one of the connected full heaps from Theorem 6.6.2.  
Then $\Gamma$ is a connected Dynkin diagram of affine Lie type, and deleting the node labeled 0 (again labeled by \cite[\S 4.8]{Kac}) produces a Dynkin diagram $\Gamma_0$ of finite Lie type $\mathcal{L}_n$ for $\mathcal{L} \in \{A,B,C,D,E\}$ and $n := |\Gamma| - 1$.  
The principal subheap of this full heap appears in an infinitely repeating motif within its Hasse diagram and is isomorphic as a $\Gamma_0$-colored poset to $\ell_n^\kappa(j)$.
Chapter 6 of \cite{Gre} contains figures of principal subheaps embedded within Hasse diagrams of full heaps.  
For example, the principal subheap displayed in \cite[Fig. 6.12]{Gre} is the connected finite $\Gamma$-colored minuscule poset $e_6^\kappa(5)$ (cf. Figure \ref{FigMinSL}(d)).
Hence by applying Theorem 8.1 we see that connected finite $\Gamma$-colored minuscule posets are embedded within connected infinite $\Gamma$-colored minuscule posets in infinitely repeating motifs.  
Each connected finite $\Gamma$-colored minuscule poset is embedded in at least one connected infinite $\Gamma$-colored minuscule poset. 


Finally, as described in Section \ref{SectionMainResults}, the $\Gamma$-colored minuscule posets are necessary and sufficient to build $P$-minuscule Kac--Moody representations from colored posets. 
Influenced by Proctor, Stembridge, and Green, this author developed the defining axioms for $\Gamma$-colored $d$-complete and $\Gamma$-colored minuscule posets as part of his doctoral work under Proctor when examining which poset coloring properties were required to satisfy which Lie bracket relations under the actions of the colored raising and lowering operators, and vice versa.
The full results of this pursuit can be found in \cite{Str,Unify}.  
Here we note that the connected finite $\Gamma$-colored minuscule poset $\ell_n^\kappa(j)$ produces the $P$-minuscule representation with basis $\{\langle F,I \rangle \ | \ (F,I) \in \mathcal{FI}(P)\}$, and this representation is isomorphic to the minuscule representation $V(\omega_j)$ in type $\mathcal{L}_n$.
The weight diagram of $V(\omega_j)$ under the standard root order on weights has structure isomorphic to $\mathcal{FI}(P)$ ordered under inclusion of ideals within the splits.  
The actions of the generators $\{x_i,y_i,h_i\}_{1 \le i \le n}$ are given by the operators $\{X_i, Y_i, H_i\}_{1 \le i \le n}$ described in Section \ref{SectionMainResults}.
For example, the four posets $a_4^\kappa(1)$, $a_4^\kappa(2)$, $a_4^\kappa(3)$, and $a_4^\kappa(4)$ in Figure \ref{FigA4Example} can be used to construct the four minuscule representations of the simple Lie algebra $\mathfrak{g} = \mathfrak{sl}_5(\mathbb{C})$ of type $A_4$ with respective highest weights $\omega_1$, $\omega_2$, $\omega_3$, and $\omega_4$. 
We note that $a_4^\kappa(1)$ and $a_4^\kappa(4)$ (and similarly $a_4^\kappa(2)$ and $a_4^\kappa(3)$) are isomorphic as $\Gamma$-colored posets, but produce non-isomorphic minuscule representations of $\mathfrak{g}$.



\end{spacing}

\section*{Acknowledgements}
I would like to thank my thesis advisor, Robert A. Proctor, for many insights into the contents of this paper.  Though much has changed from the corresponding part of my thesis to this paper, his influence is still very much present and appreciated.  I am also grateful to him specifically for help with Section \ref{SectionHistory}.  I would also like to thank Marc Besson and Sam Jeralds for helpful comments on Section \ref{SectionHistory} and the referees for several suggestions for improvement and helpful comments on exposition throughout.

\bibliographystyle{amsplain}

\end{document}